\documentclass[smallextended]{svjour3}       
\pdfoutput=1
\smartqed  

\usepackage[export]{adjustbox}
\usepackage{cleveref}
\usepackage{graphicx}

\usepackage{amsfonts,amsmath,amssymb,bm}
\usepackage{booktabs,multirow}
\usepackage{epstopdf}
\usepackage[caption=false]{subfig} 
\usepackage{algorithm}
\usepackage{algorithmic}
\usepackage{hyperref}

\makeatletter
\def\cl@chapter{}
\makeatother

\newtheorem{rem}{Remark}
\crefname{rem}{Remark}{Remark} 
\crefname{section}{Section}{Section} 
\crefname{subsection}{Subsection}{Subsection} 

\usepackage{amsopn}

\DeclareMathOperator{\Diag}{Diag}
\DeclareMathOperator{\dom}{dom}
\DeclareMathOperator{\hard}{hard}
\DeclareMathOperator{\mini}{minimize}
\DeclareMathOperator{\amin}{argmin}
\DeclareMathOperator{\Amin}{Argmin}
\DeclareMathOperator{\soft}{soft}
\DeclareMathOperator{\trace}{trace}
\DeclareMathOperator{\vect}{vect}
\DeclareMathOperator{\prox}{prox}
\DeclareMathOperator{\proj}{proj}

\DeclareMathOperator{\diff}{d}

\newcommand{\rd}{\mathbb{R}}

\newcommand{\sm}{\mathcal{S}_n}
\newcommand{\smsp}{\mathcal{S}_n^{+}}
\newcommand{\smp}{\mathcal{S}_n^{++}}

\newcommand{\Sp}[2]{\mathcal{R}_{#1}\lp {#2}\rp}
\newcommand{\tr}[1]{\trace\lp {#1}\rp}
\newcommand{\minimize}[1]{\ensuremath{\underset{\substack{{#1}}}{\mini}}\,\,}
\newcommand{\argmin}[1]{\ensuremath{\underset{\substack{{#1}}}{\amin}}\,\,}
\newcommand{\Argmin}[1]{\ensuremath{\underset{\substack{{#1}}}{\Amin}}\,\,}

\def\lp{\left(}
\def\rp{\right)}

\def\lg{\left\{}
\def\rg{\right\}}

\newcommand{\Id}{\textnormal{I}_\textnormal{d}}

\newcommand{\vectorize}[1]{\mathbf{#1}}
\newcommand{\va}{\vectorize{a}}
\newcommand{\vb}{\vectorize{b}}

\newcommand{\vd}{\vectorize{d}}
\newcommand{\ve}{\vectorize{e}}

\newcommand{\vs}{\vectorize{s}}
\newcommand{\vt}{\vectorize{t}}
\newcommand{\vx}{\vectorize{x}}
\newcommand{\vy}{\vectorize{y}}

\newcommand{\vA}{\vectorize{A}}
\newcommand{\vB}{\vectorize{B}}
\newcommand{\vC}{\vectorize{C}}

\newcommand{\vM}{\vectorize{M}}
\newcommand{\vN}{\vectorize{N}}
\newcommand{\vP}{\vectorize{P}}
\newcommand{\vQ}{\vectorize{Q}}

\newcommand{\vS}{\vectorize{S}}
\newcommand{\vT}{\vectorize{T}}
\newcommand{\vU}{\vectorize{U}}

\newcommand{\vX}{\vectorize{X}}
\newcommand{\vY}{\vectorize{Y}}

\newcommand{\iter}[3]{#1_{#2}^{#3}}

\let\orgautoref\autoref
\providecommand{\Autoref}
        {\def\equationautorefname{Equation}%
         \def\figureautorefname{Figure}%
         \def\subfigureautorefname{Figure}%
         \renewcommand\sectionautorefname{Section}%
         \def\subsectionautorefname{Section}%
         \def\subsubsectionautorefname{Section}%
         \def\Itemautorefname{Item}%
         \def\tableautorefname{Table}%
         \orgautoref}

\begin{document}

\title{A Proximal Approach for a Class of Matrix Optimization Problems}


\author{A. Benfenati$^\dagger$ \and E. Chouzenoux$^{\dagger,\ddagger}$ \\
\and J.--C. Pesquet$^\ddagger$
}


\institute{$\dagger$ $\quad$ Laboratoire d'Informatique Gaspard Monge, UMR CNRS 8049, University Paris-Est Marne-la-Vall\'ee, \email{falessandro.benfenati@esiee.fr}\\    
           $\ddagger$ $\quad$ Center for Visual Computing, INRIA Saclay and CentraleSup\'elec, University Paris-Saclay, 
              \email{emilie.chouzenoux@centralesupelec.fr}, \email{jean-christophe@pesquet.eu}
}

\date{Received: date / Accepted: date}

\maketitle

\begin{abstract}
In recent years, there has been a growing interest in mathematical models leading to the minimization, in a symmetric matrix space, of  a Bregman divergence coupled with a regularization term. We address problems of this type within a general framework  where the regularization term is split in two parts, one being a spectral function while the other is arbitrary. A Douglas--Rachford approach is proposed to address such problems and a list of proximity operators is provided allowing us to consider various choices for the fit--to--data functional and for the regularization term. Numerical experiments show the validity of this approach for solving convex optimization problems encountered in the context of sparse covariance matrix estimation. Based on our theoretical results, an algorithm is also proposed for noisy graphical lasso where a precision matrix has to be estimated in the presence of noise. The nonconvexity of the resulting objective function is dealt with a majorization--minimization approach, i.e. by building a sequence of convex surrogates and solving the inner optimization subproblems via the aforementioned  Douglas--Rachford procedure. We establish conditions for the convergence of this iterative scheme and we illustrate its good numerical performance with respect to state--of--the--art approaches.
\begin{acknowledgements}
This work was funded by the Agence Nationale de la Recherche under grant ANR-14-CE27-0001 GRAPHSIP.
\end{acknowledgements}
\end{abstract}

\section{Introduction}
\label{sec:intro}
In recent years, various applications such as shape classification models \cite{DBLP:conf/uai/2008}, gene expression \cite{Ma:2013:ADM:2494250.2494257}, model selection \cite{Banerjee:2008:MST:1390681.1390696,chandrasekaran2012}, computer vision \cite{doi:10.1093/biomet/asq060}, inverse covariance estimation \cite{Friedman07,Dempster72,Yuan09,Aspremont08,doi:10.1137/090772514}, graph estimation \cite{meinshausen2006,ravikumar2011,doi:10.1093/biomet/asm018}, social network and corporate inter-relationships analysis \cite{Aslan2016}, or brain network analysis \cite{doi:10.1137/130936397} have led to matrix variational formulations of the form:
\begin{equation}
\label{eq:prob}
\minimize{\vC\in\mathcal{S}_n}{f(\vC) -\tr{\vT\vC} + g(\vC)},
\end{equation}
where $\mathcal{S}_n$ is the set of real symmetric matrices of dimension $n\times n$, 
$\vT$ is a given $n\times n$ real matrix (without loss of generality, it will be assumed to be symmetric), 
and $f\colon \mathcal{S}_n \to ]-\infty,+\infty]$ and $g\colon \colon \mathcal{S}_n \to ]-\infty,+\infty]$ are lower-semicontinuous functions which are proper, in the sense that they are finite at
least in one point.\\
It is worth noticing that the notion of Bregman divergence \cite{bregman1967} gives a particular insight into Problem \eqref{eq:prob}. Indeed, suppose that $f$ is a convex function
differentiable on the interior of its domain $\operatorname{int}(\dom f)\neq \varnothing$.  Let us recall that, in $\mathcal{S}_n$ endowed with the Frobenius norm, the $f$-Bregman divergence between $\vC\in \mathcal{S}_n$ and $\vY\in \operatorname{int}(\dom f)$ is
\begin{equation}
\label{eq:breg0}
D^f(\vC,\vY) = f(\vC)-f(\vY)-\tr{\vT(\vC-\vY)},
\end{equation}
where $\vT = \nabla f(\vY)$ is the gradient of $f$ at $\vY$.
Hence, the original problem \eqref{eq:prob} is equivalently expressed as
\begin{equation}
\label{eq:breg}
\minimize{\vC\in\mathcal{S}_n} g(\vC)+ D^f(\vC,\vY).
\end{equation}
Solving Problem \eqref{eq:breg} amounts to computing the proximity operator of $g$ at $\vY$ with respect to the divergence $D^f$ \cite{Bauschke03,Bauschke06}
in the space $\mathcal{S}_n$. In the vector case, such kind of proximity operator has been found to be useful in a number of recent works regarding, for example, image restoration \cite{Brune2011,Benfenati13,Benfenati2015,doi:10.1137/090746379}, image reconstruction \cite{Zhang2011}, and compressive sensing problems \cite{Yin08,doi:10.1137/080725891}.\\
In this paper, it will be assumed that $f$ belongs to the class of \emph{spectral functions} \cite[Chapter 5, Section 2]{borwein2014}, i.e., for every permutation
matrix $\mathbf{\Sigma} \in \mathbb{R}^{n\times n}$,
\begin{equation}
\label{eq:fphi}
(\forall \vC \in \mathcal{S}_n) \quad f(\vC) =  \varphi(\mathbf{\Sigma}\vd),
\end{equation}
where $\varphi\colon \mathbb{R}^n \to ]-\infty,+\infty]$ is a proper lower semi-continuous convex function and
$\vd$ is a vector of eigenvalues of $\vC$.\\
Due to the nature of the problems, in many of the aforementioned applications, $g$ is a regularization function  promoting the sparsity of $\vC$. 
We  consider here a more generic class of regularization functions obtained by decomposing $g$ as $g_0 + g_1$,  where $g_0$ is a spectral function, i.e.,
for every permutation matrix $\mathbf{\Sigma} \in \mathbb{R}^{n\times n}$,
\begin{equation}
(\forall \vC \in \mathcal{S}_n) \quad g_0(\vC) =  \psi(\mathbf{\Sigma}\vd),
\label{eq:g0}
\end{equation}
with $\psi\colon \mathbb{R}^n \to ]-\infty,+\infty]$ a proper lower semi--continuous function, 
$\vd$ still denoting a vector of the eigenvalues of $\vC$,
while $g_1 \colon \mathcal{S}_n \to ]-\infty,+\infty]$ is a proper lower semi--continuous function which cannot be expressed under a spectral form. 
A very popular and useful example encompassed by our framework is the graphical lasso (GLASSO) problem, 
where $f$ is the minus log-determinant function, $g_1$ is a component--wise $\ell_1$ norm  (of the matrix elements), and $g_0 \equiv 0$. 
Various algorithms have been proposed to solve Problem \eqref{eq:prob} in this context, including the popular GLASSO algorithm \cite{Friedman07} and some of its recent variants \cite{mazumder2012}. We can also mention the dual block coordinate ascent method from \cite{Banerjee:2008:MST:1390681.1390696}, the SPICE algorithm \cite{rothman2008}, the gradient projection method in \cite{DBLP:conf/uai/2008}, the Refitted CLIME algorithm \cite{doi:10.1198/jasa.2011.tm10155}, various algorithms \cite{Aspremont08,doi:10.1137/070695915,doi:10.1137/080742531} based on Nesterov's smooth gradient approach \cite{Nesterov2005}, ADMM approaches \cite{Yuan09,NIPS2010_0109}, an inexact Newton method \cite{doi:10.1137/090772514}, and interior point methods \cite{doi:10.1093/biomet/asm018,Li2010}. A related model is addressed in \cite{Ma:2013:ADM:2494250.2494257,chandrasekaran2012}, with the additional assumption that the sought solution can be split as $\vC_1+\vC_2$, where $\vC_1$ is sparse and $\vC_2$ is low--rank. Finally, let us mention the ADMM algorithm from \cite{Zhou14}, and the incremental proximal gradient approach from \cite{Richard:2012:ESS:3042573.3042584}, both addressing Problem \eqref{eq:prob} when $f$ is the squared Frobenius norm, $g_0$ is a nuclear norm, and $g_1$ is an element--wise $\ell_1$ norm.

The main goal of this paper is to propose numerical approaches for solving Problem~\eqref{eq:prob}. Two settings will be investigated, namely (\emph{i}) $g_1 \equiv 0$, i.e. the whole cost function is a spectral one, (\emph{ii}) $g_1 \not \equiv 0$. In the former case, some general results concerning the $D^f$-proximity operator of $g_0$ are established.
In the latter case, a Douglas--Rachford optimization method is proposed, which leads us to calculate the proximity operators of several spectral functions of interest.
We then consider applications of our results to the estimation of (possibly low-rank) covariance matrices from noisy observations of multivalued random variables. 
Two variational approaches are proposed for estimating the unknown covariance matrix, depending on the prior assumptions made on it. 
We show that the cost function arising from the first formulation can be minimized through our proposed Douglas-Rachford procedure under mild assumptions on the involved regularization functions. 
The second formulation of the problem aims at preserving desirable sparsity properties of the inverse covariance (i.e., precision) matrix. 
We establish that the proposed objective function is a difference of convex terms, and we introduce a novel majorization-minimization (MM) algorithm to optimize it.

The paper is organized as follows.  \Autoref{sec:Dec} is devoted to the solution of the particular instance of Problem \eqref{eq:prob} corresponding to $g_1\equiv 0$. \Autoref{sec:DR}
describes a proximal minimization algorithm to address the problem when $g_1 \not \equiv 0$. Its implementation is discussed for a bunch of useful choices for the involved functionals. \Autoref{sec:MMDR} presents two new approaches for estimating covariance matrices from noisy data. 
Finally, in \autoref{sec:Num}, numerical experiments illustrate the applicability of the proposed methods, and its good performance with respect to the state-of-the-art, in two distinct scenarios. 

\textit{Notation:} Greek letters usually designate real numbers, bold letters designate vectors in a Euclidean space, capital bold letters indicate matrices. The $i$--th element of the vector $\vd$ is denoted by $d_i$. $\Diag(\vd)$ denotes the diagonal matrix whose diagonal elements are the components of $\vd$.
$\mathcal{D}_n$ is the cone of vectors $\vd\in \rd^n$ whose components are ordered by decreasing values.
The symbol $\vect(\vC)$ denotes the vector resulting from a column--wise ordering of the elements of matrix $\vC$. The product $\vA\otimes \vB$ denotes the classical Kronecker product of matrices $\vA$ and $\vB$.  Let $\mathcal{H}$ be a real Hilbert space endowed with an inner product $\langle \cdot,\cdot \rangle$ and a norm $\| \cdot \|$, 
the domain of a function $f\colon \mathcal{H}\to ]-\infty,+\infty]$ is 
$\dom f = \{ x \in \mathcal{H} \mid f(x) < +\infty\}$. $f$ is coercive if $\lim_{\|x\|\to +\infty} f(x) = +\infty$ and supercoercive if $\lim_{\| x \| \to +\infty} f(x)/\|x\| = +\infty$.
The Moreau subdifferential of $f$ at $x\in \mathcal{H}$ is 
$\partial f(x) = \{t \in \mathcal{H} \mid (\forall y \in  \mathcal{H}) f(y) \ge f(x)+\langle t,y-x \rangle\}$.
$\Gamma_0(\mathcal{H})$ denotes the class of lower-semicontinuous convex functions from $\mathcal{H}$ to $]-\infty,+\infty]$ with a nonempty domain (proper).
If $f\in \Gamma_0(\mathcal{H})$ is (G\^ateaux) differentiable at $x\in \mathcal{H}$, then $\partial f(x) = \{\nabla f(x)\}$ where $\nabla f(x)$ is the gradient of $f$ at $x$.
If a function $f\colon \mathcal{H}\to ]-\infty,+\infty]$ possesses a unique minimizer on a set $E\subset  \mathcal{H}$, it will be denoted by $\argmin{x\in E}{f(x)}$.
If there are possibly several minimizers, their set will be denoted by $\Argmin{x\in E}{f(x)}$. 
Given a set $E$, $\operatorname{int}(E)$ designates the interior of $E$ and $\iota_E$ denotes the indicator function of the set, which is equal to $0$ over this set and $+\infty$ otherwise.
In the remainder of the paper, the underlying Hilbert space will be $\sm$, the set of real symmetric matrices equipped with the Frobenius norm, denoted by $\|\cdot\|_{\rm F}$.
The matrix spectral norm is denoted by $\|\cdot\|_{\rm S}$, the  $\ell_1$ norm of a matrix $\vA = (A_{i,j})_{i,j}$ is  $\|\vA\|_1=\sum_{i,j}|A_{i,j}|$.
For every $p\in [1,+\infty[$,
$\Sp{p}{\cdot}$ denotes the Schatten $p$--norm, the nuclear norm being obtained when $p = 1$. 
$\mathcal{O}_n$ denotes the set of orthogonal matrices of dimension $n$ with real elements; $\smsp$ and $\smp$ denote the set of real symmetric positive semidefinite, and symmetric positive definite matrices, respectively, of dimension $n$. 
$\Id$ denotes the identity matrix whose dimension will be clear from the context. 
The soft thresholding operator $\soft_\mu$ and the hard thresholding  operator $\hard_\mu$ of parameter $\mu \in [0,+\infty[$ are given by
\begin{equation}
(\forall \xi \in \mathbb{R})\qquad
\soft_\mu(\xi) = 
\begin{cases}
\xi - \mu & \mbox{if $\xi > \mu$}\\
\xi+ \mu & \mbox{if $\xi < -\mu$}\\
0 & \mbox{otherwise}
\end{cases},\qquad
\hard_\mu(\xi) = 
\begin{cases}
\xi & \mbox{if $|\xi|>\mu$}\\
0 & \mbox {otherwise.}
\end{cases}
\end{equation}

\section{Spectral Approach}
\label{sec:Dec}
In this section, we show that, in the particular case when $g_1 \equiv 0$, Problem \eqref{eq:prob}  reduces to the optimization of a function defined on $\mathbb{R}^n$.
Indeed, the problem then reads: 
\begin{equation}
\label{eq:spectprob}
\minimize{\vC\in\sm}  {f(\vC) -\tr{\vT\vC} + g_0(\vC)},
\end{equation}
where the spectral forms of $f$ and $g_0$ allow us to take advantage of the eigendecompositions of $\vC$ and $\vT$
in order to simplify the optimization problem, as stated below.
\begin{theorem}
\label{thh:P}
Let $\vt\in \mathbb{R}^n$ be a vector of eigenvalues of $\vT$ and let  $\vU_\vT\in\mathcal{O}_n$ be such that $\vT = \vU_T\Diag(\vt)\vU_T^\top$.
Let $f$ and $g_0$ be functions satisfying  \eqref{eq:fphi} and \eqref{eq:g0}, respectively,
where $\varphi$ and $\psi$ are lower-semicontinuous functions.
Assume that  $\dom \varphi \cap \dom \psi \neq \varnothing$ and that the function $\vd \mapsto \varphi(\vd) -\vd^\top\vt + \psi(\vd)$ is coercive.
Then a solution to Problem \eqref{eq:spectprob} exists, which is given by
\begin{equation}\label{e:vCoptspec}
\widehat{\vC} = \vU_\vT \Diag(\widehat{\vd}) \vU_\vT^\top
\end{equation}
where $\widehat{\vd}$ is any solution to the following problem: 
\begin{equation}
\label{eq:spectprob1}
\minimize{\vd\in\rd^n}  \varphi(\vd) -\vd^\top\vt + \psi(\vd).
\end{equation}
\end{theorem}
For the sake of clarity, before establishing this result, we recall two useful lemmas from linear algebra.
\begin{lemma}{\rm \cite[Chapter 9, Sec. H, p. 340]{marshall11}}
\label{lemma:1} Let $\vC\in\sm$ and let $\vd\in \mathcal{D}_n$ be a vector of ordered eigenvalues of this matrix.
Let $\vT\in\sm$ and let $\vt\in \mathcal{D}_n$ be a vector of ordered eigenvalues of this matrix.
The following inequality holds:
\begin{equation}
\tr{\vC  \vT  }  \le \vd^\top \vt.
\end{equation}
In addition, the upper bound is reached if and only if
$\vT$ and $\vC$ share the same eigenbasis, i.e. there exists $\vU \in \mathcal{O}_n$ such that
$\vC=\vU \Diag(\vd)\vU^\top$ and $\vT = \vU \Diag(\vt) \vU^\top$.
\end{lemma}
The subsequent lemma is also known as the \emph{rearrangement inequality}:
\begin{lemma} {\rm \cite[Section 10.2, Theorem 368]{hardy1952inequalities}} \label{lemma:2}
Let $\va\in \mathcal{D}_n$ and $\vb \in \mathcal{D}_n$.
Then, for every permutation matrix $\vP$ of dimension $n\times n$,
\begin{equation}
\va^\top\vP \vb \le \va^\top \vb.
\end{equation}
\end{lemma}
We are now ready to prove \autoref{thh:P}.
\begin{proof}[\autoref{thh:P}] 
Due to the assumptions made on $f$ and $g_0$, Problem \eqref{eq:spectprob} can be reformulated as
$$
\minimize{\vd\in\mathcal{D}_n, \vU_\vC \in\mathcal{O}_n} \varphi(\vd) -\tr{\vU_\vC\Diag(\vd)\vU_\vC^\top\vT} +\psi(\vd).
$$
According to the first claim in \autoref{lemma:1},
$$
 \inf_{\vd\in\mathcal{D}_n, \vU_\vC \in\mathcal{O}_n}\varphi(\vd) -\tr{\vU_\vC\Diag(\vd)\vU_\vC^\top\vT} +\psi(\vd) \geq \inf_{\vd\in\mathcal{D}_n} \varphi(\vd) -\vd^\top \widetilde{\vt} +\psi(\vd), 
$$
where $\widetilde{\vt}\in\mathcal{D}_n$ is the vector of ordered eigenvalues of  $\vT=\widetilde{\vU}_\vT\Diag(\widetilde{\vt})\widetilde{\vU}_\vT^\top$ with $\widetilde{\vU}_\vT \in \mathcal{O}_n$. 
In addition, the last claim in \autoref{lemma:1} allows us to conclude that the lower bound is attained when $\vU_\vC = \widetilde{\vU}_\vT$.
This proves that
\begin{equation}\label{e:infinf1}
\inf_{\vC\in\sm}  {f(\vC) -\tr{\vT\vC} + g_0(\vC)} = \inf_{\vd\in\mathcal{D}_n} \varphi(\vd) -\vd^\top \widetilde{\vt} +\psi(\vd).
\end{equation}
Let us now show that ordering the eigenvalues is unnecessary for our purposes.
Let $\vt\in \mathbb{R}^n$ be a vector of non necessarily ordered eigenvalues of $\vT$.
Then, $\vT=\vU_\vT\Diag(\vt)\vU_\vT^\top$ with $\vU_\vT \in \mathcal{O}_n$ and
there exists a permutation matrix $\vQ$ such that $\vt = \vQ \widetilde{\vt}$.
For every vector $\vd \in \mathcal{D}_n$ and for every permutation matrix $\vP$ of dimension $n\times n$, we have then
\begin{align}
\varphi(\vP\vd) - (\vP\vd)^\top \vt+ \psi(\vP\vd)
=&\,\varphi(\vP\vd) - (\vP\vd)^\top \vQ\widetilde{\vt}+ \psi(\vP\vd)\\
=&\,  \varphi(\vd) - (\vQ^\top\vP\vd)^\top \widetilde{\vt}+ \psi(\vd)\nonumber\\
\geq &\, \varphi(\vd) - \vd^\top \widetilde{\vt}+ \psi(\vd),\nonumber
\end{align}
where the last inequality is a direct consequence of \autoref{lemma:2}. In addition, the equality is obviously reached
if $\vP = \vQ$. Since every vector in $\mathbb{R}^n$ can be expressed as permutation of a vector in $ \mathcal{D}_n$,
we deduce that
\begin{equation}\label{e:infinf2} 
\inf_{\vd\in \mathbb{R}^n} \varphi(\vd) - \vd^\top \vt+ \psi(\vd) = \inf_{\vd \in \mathcal{D}_n} \varphi(\vd) - \vd^\top \widetilde{\vt}+ \psi(\vd).
\end{equation}
Altogether, \eqref{e:infinf1} and \eqref{e:infinf2} lead to
\begin{equation}
\inf_{\vC\in\sm}  {f(\vC) -\tr{\vT\vC} + g_0(\vC)} = \inf_{\vd\in \mathbb{R}^n} \varphi(\vd) - \vd^\top \vt+ \psi(\vd).
\end{equation}
Since the function $\vd \mapsto  \varphi(\vd) - \vd^\top \vt+ \psi(\vd)$ is proper, lower-semicontinuous, and coercive, it follows from \cite[Theorem 1.9]{Rockfellar97} that there exists $\widehat{\vd}\in \mathbb{R}^n$ such that
\begin{equation}
\varphi(\widehat{\vd}) - \widehat{\vd}^\top \vt+ \psi(\widehat{\vd}) = \inf_{\vd\in \mathbb{R}^n} \varphi(\vd) - \vd^\top \vt+ \psi(\vd).
\end{equation}
In addition, it is easy to check that if $\widehat{\vC}$ is given by \eqref{e:vCoptspec} then
\begin{equation}
f(\widehat{\vC}) -\tr{\vT\widehat{\vC}} + g_0(\widehat{\vC}) = \varphi(\widehat{\vd}) - \widehat{\vd}^\top \vt+ \psi(\widehat{\vd}),
\end{equation}
which yields the desired result.\qed
\end{proof}
Before deriving a main consequence of this result, we need to recall some definitions from convex analysis \cite[Chapter 26]{Rockfellar70} \cite[Section~3.4]{Bauschke03}:
\begin{definition}
Let $\mathcal{H}$ be a finite dimensional real Hilbert space with norm $\| \cdot \|$ and scalar product $\langle \cdot, \cdot \rangle$. Let $h\colon \mathcal{H}\to ]-\infty,+\infty]$ be a
proper convex function.
\begin{itemize}
\item[$\bullet$] $h$ is essentially smooth if $h$ is differentiable on $\operatorname{int}(\dom h)\neq \varnothing$ and\linebreak $\lim_{n\to +\infty} \|\nabla h(x_n)\| = +\infty$
for every sequence $(x_n)_{n\in \mathbb{N}}$ of $\operatorname{int}(\dom h)$ converging to a point on the boundary of $\dom h$.
\item[$\bullet$]  $h$ is essentially strictly convex if $h$ is strictly convex on every convex subset of the domain of its subdifferential.
\item[$\bullet$]  $h$ is a Legendre function if it is both essentially smooth and essentially strictly convex.
\item[$\bullet$]  If $h$ is differentiable on $\operatorname{int}(\dom h)\neq \varnothing$, the $h$-Bregman divergence is the function
$D^h$ defined on $\mathcal{H}^2$ as
\begin{multline}
(\forall (x,y)\in \mathcal{H}^2)\\
D^h(x,y) = \begin{cases}
h(x)-h(y) - \langle \nabla h(y),x-y  \rangle & \mbox{if $y \in \operatorname{int}(\dom f)$}\\
+\infty & \mbox{otherwise.}
\end{cases}
\end{multline}
\item[$\bullet$]  Assume that $h$ is a lower-semicontinuous Legendre function and that $\ell$ is a lower-semicontinuous convex function
such that $\operatorname{int}(\dom h) \cap \dom \ell \neq \varnothing$ and either $\ell$ is bounded from below 
or $h+\ell$ is supercoercive. Then, the $D^h$-proximity operator of $\ell$ is
\begin{align}
\prox^h_\ell \colon  \operatorname{int}(\dom h) &\to \operatorname{int}(\dom h) \cap \dom \ell\\
y &\mapsto \argmin{x\in \mathcal{H}}{\ell(x)+D^h(x,y)}.\nonumber
\end{align}
\end{itemize}
\end{definition}
In this definition, when $h = \|\cdot\|^2/2$, we recover the classical definition of the proximity operator in \cite{Moreau65},
which is defined over $\mathcal{H}$, for every function $\ell \in \Gamma_0(\mathcal{H})$, and that will be simply
denoted by $\prox_\ell$. 

We will also need the following result:
\begin{lemma}\label{lem:dom}
Let $f$ be a function satisfying  \eqref{eq:fphi} where $\varphi\colon \mathbb{R}^n \to ]-\infty,+\infty]$. Let $\vC \in \sm$ and let $\vd \in \mathbb{R}^n$ be a vector of eigenvalues of this matrix.
The following hold:
\begin{enumerate}
\item[(i)] $\vC \in \dom f$ if and only if $\vd \in \dom \varphi$;
\item[(ii)] $\vC \in \operatorname{int}(\dom f)$ if and only if $\vd \in \operatorname{int}(\dom \varphi)$.
\end{enumerate}
\end{lemma}
\begin{proof}
(i) obviously holds since $f$ is a spectral function.\\
Let us now prove (ii).
If $\vC \in \operatorname{int}(\dom f)$, then $\vd \in \dom \varphi$. In addition, there exists $\rho \in ]0,+\infty[$ such that,
for every $\vC'\in \sm$, if $\| \vC'-\vC \|_{\rm F} \le \rho$, then $\vC'\in \dom f$.
Let $\vU_\vC \in \mathcal{O}_n$ be such that $\vC = \vU_\vC \Diag(\vd) \vU_\vC^\top$ and let
us choose $\vC' = \vU_\vC \Diag(\vd') \vU_\vC^\top$ with $\vd' \in \mathbb{R}^n$. Since $\vC$ and $\vC'$ share the same eigenbasis,
\begin{equation}
\| \vC'-\vC \|_{\rm F} = \|  \vd' - \vd \|.
\end{equation}
Hence, for any $\vd' \in \mathbb{R}^n$ such that $\|  \vd' - \vd \| \le \rho$, $\vC'\in \dom f$, hence $\vd'\in \dom \varphi$.
This shows that $\vd \in \operatorname{int}(\dom \varphi)$.\\
Conversely, let us assume that $\vd = (d_i)_{1\le i \le n} \in \operatorname{int}(\dom \varphi)$.
Without loss of generality, it can be assumed that $\vd  \in \mathcal{D}_n$. There thus exists $\rho \in ]0,+\infty[$ such that
for every $\vd' = (d'_i)_{1\le i \le n} \in \mathcal{D}_n$, if 
\begin{equation}
(\forall i \in \{1,\ldots,n\}) \qquad | d'_i- d_i | \le \rho,
\end{equation}
then $\vd' \in \dom \varphi$. 
Furthermore, let $\vC'$ be any matrix in $\sm$ such that
\begin{equation}
\| \vC' - \vC \|_{\rm F} \le \rho
\end{equation}
and let $\vd' = (d'_i)_{1\le i \le n} \in \mathcal{D}_n$ be a vector of eigenvalues of  $\vC$.
It follows from Weyl's inequality \cite{marshall11} that
\begin{equation}
(\forall i \in \{1,\ldots,n\}) \qquad |d'_i - d_i | \le \| \vC' - \vC \|_{\rm S} \le \| \vC' - \vC \|_{\rm F} \le \rho.
\end{equation}
We deduce that $\vd' \in \dom \varphi$ and, consequently $\vC' \in  \dom f$. This shows 
that $\vC \in \operatorname{int}(\dom f)$. \qed
\end{proof}

As an offspring of \autoref{thh:P}, we then get:
\begin{corollary}\label{co:proxdiv}
Let $f$ and $g_0$ be functions satisfying  \eqref{eq:fphi} and \eqref{eq:g0}, respectively, where $\varphi\in \Gamma_0(\mathbb{R}^n)$ is a Legendre function,
$\psi \in \Gamma_0(\mathbb{R}^n)$, $\operatorname{int}(\dom \varphi) \cap \dom \psi\neq \varnothing$, and either $\psi$ is bounded from below 
or $\varphi+\psi$ is supercoercive. Then, the $D^f$-proximity operator of $g_0$ is defined at  every $\vY\in\sm$ such that
 $\vY = \vU_\vY\Diag(\vy)\vU_\vY^\top$ with $\vU_\vY\in\mathcal{O}_n$ and $\vy \in \operatorname{int}(\dom \varphi)$, and it is expressed as
\begin{align}
\prox^f_{g_0}(\vY) = \vU_\vY \Diag(\prox^\varphi_\psi(\vy)) \vU_\vY^\top.
\end{align}
\end{corollary}
\begin{proof}
According to the properties of spectral functions \cite[Corollary 2.7]{Lewis96}, 
\begin{equation}
\varphi \in \Gamma_0(\mathbb{R}^n) \;\mbox{(resp. $\psi \in \Gamma_0(\mathbb{R}^n)$)} \quad \Rightarrow \quad f \in \Gamma_0(\sm)  
\:\mbox{(resp. $g_0 \in \Gamma_0(\sm)$)}.
\end{equation}
In addition, according to \cite[Corollaries 3.3\&3.5]{Lewis96}, since $\varphi$ is a Legendre function, $f$ is a Legendre function.
It is also straightforward to check that, when $\psi$ is lower bounded, then $g_0$ is lower bounded and, when $\varphi+\psi$ is supercoercive,
then $f+g_0$ is supercoercive.  It also follows from \autoref{lem:dom} that
$\operatorname{int}(\dom \varphi) \cap \dom \psi\neq \varnothing \Leftrightarrow \operatorname{int}(\dom f) \cap \dom g_0 \neq \varnothing$.
%
%

The above results show that the $D^f$-proximity operator of $g_0$ is properly defined as follows:
\begin{align}
\prox^f_{g_0} \colon  \operatorname{int}(\dom f) &\to \operatorname{int}(\dom f) \cap \dom g_0\\
\vY &\mapsto \argmin{\vC\in \sm}{g_0(\vC)+D^f(\vC,\vY)}.\nonumber
\end{align}
This implies that computing the $D^f$-proximity operator of $g_0$
at $\vY \in \operatorname{int}(\dom f)$ amounts to finding
the unique solution to Problem~\eqref{eq:spectprob} where $\vT = \nabla f(\vY)$.
Let $\vY = \vU_\vY \Diag(\vy) \vU_\vY^\top$ with $\vU_\vY \in \mathcal{O}_n$ and $\vy \in \mathbb{R}^n$.
By \autoref{lem:dom}(ii), $\vY \in \operatorname{int}(\dom f) \Leftrightarrow \vy\in \operatorname{int}(\dom(\varphi))$ and,
according to \cite[Corollary 3.3]{Lewis96}, $\vT = \vU_\vY \Diag(\vt) \vU_\vY^\top$ with $\vt = \nabla \varphi(\vy)$.

Furthermore, as $\varphi$ is essentially strictly convex, it follows from \cite[Theorem 5.9(ii)]{Bauscke-Borwein01} that
$\vt=\nabla \varphi(\vy) \in \operatorname{int}(\dom f^*)$, which according to \cite[Theorem 14.17]{Bauschke:2017} is equivalent to
the fact that $\vd \mapsto \varphi(\vd) - \vd^\top \vt$ is coercive. So, if $\psi$ is lower-bounded, 
$\vd\mapsto \varphi(\vd) - \vd^\top \vt + \psi(\vd)$ is coercive. The same conclusion obviously holds if
$\varphi+\psi$ is supercoercive. This shows that the assumptions of \autoref{thh:P} are met.
Consequently, applying this theorem yields  
\begin{equation}
\prox^f_{g_0}(\vY) = \vU_\vY \Diag(\widehat{\vd})  \vU_\vY^\top,
\end{equation}
where $\widehat{\vd}$ minimizes
\begin{equation}
\vd \mapsto \varphi(\vd) - \vd^\top \vt + \psi(\vd)
 \end{equation}
 or, equivalently,
\begin{equation}
\vd \mapsto  \psi(\vd)+D^\varphi(\vd,\vy).
 \end{equation}
This shows that $\widehat{\vd} = \prox^\varphi_\psi(\vy)$.\qed
\end{proof}
\begin{rem}
\autoref{co:proxdiv} extends known results concerning the case when $f=$\linebreak $\|\cdot\|_{\rm F}/2$ \cite{Cai10}.
A rigorous derivation of the proximity operator of spectral functions in $\Gamma_0(\sm)$ 
for the standard Frobenius metric can be found in \cite[Corollary 24.65]{Bauschke:2017}.
Our proof allows us to recover a similar result by adopting a more general approach. In particular, it is worth noticing that
\autoref{thh:P}  does not require any convexity assumption.
\end{rem}

\section{Proximal Iterative Approach}
\label{sec:DR}
Let us now turn to the more general case of the resolution of Problem \eqref{eq:prob} when $f\in \Gamma_0(\mathcal{S}_n)$ and $g_1 \not\equiv 0$. Proximal splitting approaches for finding a minimizer
of a sum of non-necessarily smooth functions have 
attracted a large interest in the last years \cite{combettes:hal-00643807,Parikh14,Komodakis15,Burger2016}. In these methods, the functions can be dealt with either via their gradient or their proximity operator depending on their differentiability properties. In this section, we first list a number of proximity operators of scaled versions of $ f  - \tr{\vT\, \cdot} +g_0$, where $f$ and $g_0$, satisfying \eqref{eq:fphi} and \eqref{eq:g0}, are chosen among several options that can be useful in a wide range of practical scenarios. Based on these results, we then propose a proximal splitting Douglas-Rachford algorithm to solve Problem \eqref{eq:prob}.  

\subsection{Proximity Operators}
\label{sub:prox}
By definition, computing the proximity operator of $\gamma \lp f  - \tr{\vT\, \cdot} +g_0\rp$ with $\gamma \in ]0,+\infty[$
at $\overline{\vC}  \in \sm$ amounts to 
find a minimizer of the function
\begin{equation}
\label{eq:prox}
\vC \mapsto f(\vC) - \tr{\vT \vC} +g_0(\vC) +\frac{1}{2\gamma}\|\vC-\overline{\vC}\|_{\rm F}^2
\end{equation}
over $\mathcal{S}_n$.
The (possibly empty) set of such minimizers is denoted by\linebreak $\operatorname{Prox}_{\gamma \lp f  - \tr{\vT\, \cdot} +g_0\rp}(\overline{\vC})$.
 As pointed out in Section \ref{sec:Dec}, if $f + g_0 \in \Gamma_0(\mathcal{S}_n)$ then this set is a singleton 
 $\{ \prox_{\gamma \lp f  - \tr{\vT\, \cdot} +g_0\rp}(\overline{\vC})\}$. We have the following characterization of this proximity operator:
 \begin{proposition}\label{prop:proxDR}
Let $\gamma \in ]0,+\infty[$ and $\overline{\vC}  \in \sm$.
 Let $f$ and $g_0$ be functions satisfying  \eqref{eq:fphi} and \eqref{eq:g0}, respectively, where $\varphi\in \Gamma_0(\mathbb{R}^n)$
 and $\psi$ is a lower-semicontinuous function such that  $\dom \varphi \cap \dom \psi \neq \varnothing$. 
Let $\bm{\lambda}\in \mathbb{R}^n$ and $\vU\in\mathcal{O}_n$ be such that $\overline{\vC}+\gamma \vT = \vU\Diag(\bm{\lambda}) \vU^\top$.
\begin{enumerate}
\item[(i)] If $\psi$ is lower bounded by an affine function then $\operatorname{Prox}_{\gamma\lp \varphi + \psi\rp}\lp\bm{\lambda}\rp  \neq \varnothing$ and, for every 
$\widehat{\bm{\lambda}} \in \operatorname{Prox}_{\gamma\lp \varphi + \psi\rp}\lp\bm{\lambda}\rp$,
\begin{equation}
\label{eq:proxfg0ncs}
\vU \Diag(\widehat{\bm{\lambda}})\vU^\top \in \operatorname{Prox}_{\gamma \lp f  - \tr{\vT\, \cdot} +g_0\rp}(\overline{\vC}).
\end{equation}
\item[(ii)]  If $\psi$ is convex, then
\begin{equation}
\label{eq:proxfg0}
\prox_{\gamma \lp f  - \tr{\vT\, \cdot} +g_0\rp}(\overline{\vC})=
 \vU \Diag\Big(\!\prox_{\gamma\lp \varphi + \psi\rp}\lp\bm{\lambda}\rp\!\Big)\vU^\top.
\end{equation}
\end{enumerate}
 \end{proposition}
 \begin{proof}
(i) Since it has been assumed that $f$ and $g_0$ are spectral functions, we have
\begin{equation}
(\forall \vC\in \sm) \quad f(\vC) + g_0(\vC) = \varphi(\vd) + \psi(\vd),
\end{equation}
where $\vd\in\mathbb{R}^n$ is a vector of the eigenvalues of $\vC$. 
It can be noticed that minimizing \eqref{eq:prox} is obviously equivalent to minimize
$\widetilde{f}- \gamma^{-1} \tr{(\overline{\vC}+\gamma \vT} \,\cdot)+g_0$
where $\widetilde{f} = f+ \|\cdot\|_{\rm F}^2/(2\gamma)$.
Then 
\begin{equation}
\widetilde{f}(\vC) = \widetilde{\varphi}(\vd),
\end{equation}
where $\widetilde{\varphi} = \varphi +\|\cdot\|^2/(2\gamma)$.
Since we have assumed that $\varphi \in \Gamma_0(\mathbb{R}^n)$,
$\widetilde{\varphi}$ is proper, lower-semicontinuous, and strongly convex.
As $\psi$ is lower bounded by an affine function, it follows that 
\begin{equation}
\label{eq:deftildevarphipluspsi}
\vd \mapsto \widetilde{\varphi}(\vd)-\gamma^{-1} \bm{\lambda}^\top \vd +\psi(\vd)
\end{equation}
is lower bounded by a strongly convex function and it is thus coercive.
In addition, $\dom \widetilde{\varphi} = \dom \varphi$, hence $\dom \widetilde{\varphi}\cap \dom \psi \neq \varnothing$.
Let us now apply \autoref{thh:P}.
Let $\widehat{\bm{\lambda}}$ be a minimizer of \eqref{eq:deftildevarphipluspsi}.
It can be claimed that  $\widehat{\vC} = \vU \Diag(\widehat{\bm{\lambda}})\vU^\top$
is a minimizer of \eqref{eq:prox}. On the other hand, minimizing \eqref{eq:deftildevarphipluspsi} is equivalent to
minimize $\gamma(\varphi+\psi) + \frac12 \|\cdot - \bm{\lambda}\|^2$, which shows that 
$\widehat{\bm{\lambda}} \in \operatorname{Prox}_{\gamma\lp \varphi + \psi\rp}\lp\bm{\lambda}\rp$.

(ii) If $\psi \in \Gamma_0(\mathbb{R}^n)$, then it is lower bounded by an affine function \cite[Theorem~9.20]{Bauschke:2017}.
Furthermore, $\varphi+\psi \in \Gamma_0(\mathbb{R}^n)$ 
and the proximity operator of $\gamma\lp \varphi + \psi\rp$ is thus single valued. On the other hand, we also have 
$\gamma \lp f  - \tr{\vT\, \cdot}\right.$ $\left.+g_0\rp \in \Gamma_0(\sm)$ \cite[Corollary 2.7]{Lewis96}, and the proximity operator of this function
is single valued too. The result directly follows from (i).\qed

\end{proof}
We will next focus on the use of \autoref{prop:proxDR} for three choices for $f$, namely the classical squared Frobenius norm, the minus $\log\det$ functional, and the Von Neumann entropy, each choice being coupled with various possible choices for $g_0$.

\subsubsection{Squared Frobenius Norm}
\label{subsub:fro}
A suitable choice in Problem \eqref{eq:prob} is $f = \|\cdot\|_{\rm{F}}^2/2$~\cite{Zhou14,Richard:2012:ESS:3042573.3042584,chartrand12}.
The squared Froebenius norm is the spectral function associated with the function $\varphi = \|\cdot\|^2/2$.

It is worth mentioning that this choice for $f$ allows us to rewrite the original Problem \eqref{eq:prob} under the form \eqref{eq:breg},
where
\begin{equation}
\label{eq:BregFro}
\big(\forall (\vC,\vY)\in \sm^2\big)\quad
D^f(\vC,\vY)
= \frac12 \|\vC-\vY\|_{\rm F}^2.
\end{equation}
We have thus re-expressed Problem \eqref{eq:prob} as the determination of a proximal point of function $g$ 
at $\vT$ in the Frobenius metric.

Table \ref{tab:froeb} presents several examples of spectral functions $g_0$
and the expression of the proximity operator 
of $\gamma (\varphi+\psi)$  with $\gamma \in ]0,+\infty[$. These expressions were established by using the properties of proximity operators
of functions defined on $\mathbb{R}^n$ (see \cite[Example 4.4]{0266-5611-23-4-008} and \cite[Tables 10.1 and 10.2]{combettes:hal-00643807}).

\begin{center}
\small
\begin{table}[htbp]
\caption{Proximity operators of $\gamma (\frac12 \|\cdot\|_{\rm F}^2+g_0)$ with $\gamma > 0$ evaluated at symmetric matrix with vector of eigenvalues
$\bm{\lambda}=(\lambda_i)_{1\le i \le n}$.
For the inverse Schatten penalty, the function is set to $+\infty$ when the argument
$\vC$ is not positive definite.
$E_1$ denotes the set of matrices in $\sm$ with Frobenius norm less than or equal to $\alpha$ and $E_2$ the set of matrices in $\sm$ with eigenvalues between $\alpha$
and $\beta$.
In the last line, the $i$-th component of the proximity operator is obtained by searching among the nonnegative roots of a third order polynomial
those minimizing $\lambda'_i \mapsto \frac12 (\lambda'_i -|\lambda_i|)^2+\gamma\big(\frac12 (\lambda'_i)^2 + \mu \log((\lambda'_i)^2+\varepsilon)\big)$.
}
\label{tab:froeb}
\resizebox{\columnwidth}{!}{%
\begin{tabular}{c|c}

\toprule
$g_0(\vC), \, \mu>0$ 
& $\prox_{\gamma(\varphi+ \psi)}(\bm{\lambda})$		\\
\toprule
\toprule
Nuclear norm 					
&   \multirow{2}{*}{$\lp\soft_{\frac{\mu\gamma}{\gamma+1}}\lp\frac{\lambda_i}{\gamma+1}\rp\rp_{1\le i \le n}$}\\
$\mu \mathcal{R}_1(\vC)$ 				
&  \\
\midrule
Frobenius norm 			
&\multirow{2}{*}{$\lp 1-\frac{\gamma \mu}{\| \bm{\lambda} \|}\rp \frac{\bm{\lambda}}{1+\gamma}$ if $\| \bm{\lambda} \| > \gamma\mu$ and $\bm{0}$ otherwise}

\\
$\mu \|\vC\|_{\rm F}$
& \\
\midrule
Squared Frobenius norm 			
&\multirow{2}{*}{$\displaystyle\frac{\bm{\lambda}}{1+\gamma\lp1+2\mu\rp}$} \\
$\mu \|\vC\|_{\rm F}^2$
& \\
\midrule
Schatten $3$--penalty 			
&\multirow{2}{*}{$\displaystyle (6\gamma\mu)^{-1}\lp\operatorname{sign}\lp\lambda_i\rp\sqrt{(\gamma+1)^2+ 12|\lambda_i|\gamma\mu}-\gamma-1\rp_{1\le i \le n}$} \\
$\mu \mathcal{R}_3^3(\vC)$ 
& \\
\midrule
Schatten $4$--penalty 			
&\multirow{3}{*}{$(8\gamma\mu)^{-1/3}\lp\displaystyle \sqrt[3]{\lambda_i+\sqrt{\lambda_i^2+\zeta}}+ \sqrt[3]{\lambda_i-\sqrt{\lambda_i^2+\zeta}}\rp_{1\le i \le n}$ with $\zeta = \frac{\lp\gamma+1\rp^3}{27\gamma\mu}$} \\
\multirow{2}{*}{$\mu \mathcal{R}_4^4(\vC)$} 
& \\
&\\
\midrule
Schatten $4/3$--penalty 			
&\multirow{2}{*}{$\frac{1}{1+\gamma}\lp \lambda_i+ \frac{4\gamma \mu}{3\sqrt[3]{2(1+\gamma)}} \Big( \sqrt[3]{\sqrt{\lambda_i^2+\zeta}-\lambda_i}- \sqrt[3]{\sqrt{\lambda_i^2+\zeta}+\lambda_i} \Big) \rp_{1\le i \le n}  $} \\
\multirow{3}{*}{$\mu \mathcal{R}_{4/3}^{4/3}(\vC)$} 
& \multirow{3}{*}{with $\zeta = \frac{256(\gamma \mu)^3}{729(1+\gamma)}$}\\
&\\
\midrule
Schatten $3/2$--penalty 			
&\multirow{3}{*}{$\frac{1}{1+\gamma}\lp \lambda_i + \frac{9\gamma^2\mu^2}{8(1+\gamma)} \operatorname{sign}(\lambda_i) \Big(1- \sqrt{1+ \frac{16(1+\gamma)}{9\gamma^2\mu^2} 
|\lambda_i |}\Big)\rp_{1\le i \le n} $} \\
\multirow{2}{*}{$\mu \mathcal{R}_{3/2}^{3/2}(\vC)$} 
& \\
&\\
\midrule
Schatten $p$--penalty			
& $\big(\operatorname{sign}(\lambda_i)d_i\big)_{1\le i \le n}$\\ 
$\mu \mathcal{R}_p^p(\vC)$, $p \geq 1$	
& with $(\forall i \in \{1,\ldots,n\})$ $d_i \ge 0$ and $\mu\gamma p d_i^{p-1} +(\gamma+1)d_i=\lambda_i$ \\
\midrule
Inverse Schatten $p$--penalty			
& $\big(d_i\big)_{1\le i \le n}$\\ 
$\mu \mathcal{R}_p^p(\vC^{-1})$, $p > 0$	
& with $(\forall i \in \{1,\ldots,n\})$ $d_i > 0$ and $ (\gamma+1)d_i^{p+2} -\lambda_i d_i^{p+1}=\mu\gamma p $ \\
\midrule
Bound on the Frobenius norm 			
&\multirow{2}{*}{$\displaystyle\alpha\frac{\bm{\lambda}}{\| \bm{\lambda} \|}$ if $\| \bm{\lambda} \| > \alpha(1+\gamma)$ and $\displaystyle\frac{\bm{\lambda}}{1+\gamma}$ otherwise, $\alpha \in [0,+\infty[$}

\\
$\iota_{E_1}(\vC)$
& \\
\midrule
Bounds on eigenvalues	
& \multicolumn{1}{c}{\multirow{2}{*}{$\lp\min(\max(\lambda_i / (\gamma+1),\alpha),\beta)\rp_{1\le i \le n}$, $[\alpha,\beta] \subset [-\infty,+\infty]$}}\\
$\iota_{E_2}(\vC)$  & \\
\midrule\midrule
Rank 				
&  	 \multicolumn{1}{c}{\multirow{2}{*}{$\lp\hard_{\sqrt{\frac{2\mu\gamma}{1+\gamma}}}\lp \displaystyle\frac{\lambda_i}{1+\gamma}\rp\rp_{1\le i \le n}$}}\\
$\mu \operatorname{rank}(\vC)$ 				
& \\

\midrule
 Cauchy 				
 & $\in \big\{(\operatorname{sign}(\lambda_i)d_i)_{1\le i \le n}\mid (\forall i \in \{1,\ldots,n\})\; d_i \ge 0$ and\\
$\mu \log\det(\vC^2+\varepsilon \Id)$, $\varepsilon > 0$	
& $\qquad(\gamma+1)d_i^3-|\lambda_i| d_i^2 + \big(2\gamma\mu+\varepsilon(\gamma+1)\big) d_i = |\lambda_i|\varepsilon \big\}$\\
\bottomrule
\end{tabular}
}
\end{table} 
\end{center}

\begin{rem}
Another option for $g_0$ is to choose it equal to $\mu \|\cdot\|_{\rm S}$ where $\mu \in ]0,+\infty[$. 
For every $\gamma \in ]0,+\infty[$, we have then
\begin{equation}
(\forall \bm{\lambda} \in \mathbb{R}^n)\qquad 
\prox_{\gamma\lp \varphi + \psi\rp}\lp\bm{\lambda}\rp = \prox_{\frac{\mu \gamma}{1+\gamma}\|\cdot\|_{+\infty}}\lp\frac{\bm{\lambda}}{1+\gamma}\rp,
\end{equation}
where $\|\cdot\|_{+\infty}$ is the infinity norm of $\mathbb{R}^n$.
By noticing that $\|\cdot\|_{+\infty}$ is the conjugate function of the indicator function of $B_{\ell^1}$, the unit $\ell^1$ ball centered at 0 of $\mathbb{R}^n$,
and using Moreau's decomposition formula, \cite[Proposition 24.8(ix)]{Bauschke:2017} yields
\begin{equation}
(\forall \bm{\lambda} \in \mathbb{R}^n)\qquad 
\prox_{\gamma\lp \varphi + \psi\rp}\lp\bm{\lambda}\rp =\frac{1}{1+\gamma}\lp \bm{\lambda}-\mu \gamma
\proj_{B_{\ell^1}}\lp \frac{\bm{\lambda}}{\mu\gamma} \rp\rp.
\end{equation}
The required projection onto $B_{\ell^1}$ can be computed through efficient algorithms \cite{doi:10.1137/080714488,Condat2016}.

\end{rem}

\subsubsection{Logdet Function}
\label{subsub:log}
Another popular choice for $f$ is the negative logarithmic determinant function~\cite{DBLP:conf/uai/2008,NIPS2010_0109,Ma:2013:ADM:2494250.2494257,meinshausen2006,Banerjee:2008:MST:1390681.1390696,Friedman07,doi:10.1093/biomet/asm018,chandrasekaran2012}, which is defined as follows 
\begin{equation}
\label{eq:logdetfunc}
(\forall \vC \in \sm) \quad 
f(\vC) = 
\begin{cases}
-\log\det(\vC) & \mbox{if $\vC \in\smp$}\\
+\infty & \mbox{otherwise.}
\end{cases}
\end{equation}
The above function satisfies property \eqref{eq:g0} with
\begin{equation}
\label{eq:logdetfuncspec}
\big(\forall \bm{\lambda} = (\lambda_i)_{1\le i \le n} \in \mathbb{R}^n\big) \qquad 
\varphi(\bm{\lambda}) = 
\begin{cases}
-\displaystyle\sum_{i=1}^n\log(\lambda_i) & \mbox{if $\bm{\lambda} \in ]0,+\infty[^n$}\\
+\infty & \mbox{otherwise.}
\end{cases}
\end{equation}
Actually, for a given positive definite matrix, the value of function \eqref{eq:logdetfunc} simply reduces to the Burg entropy of  its eigenvalues. Hereagain, 
if $\vY \in \smp$ and $\vT = -\vY^{-1}$,
we can rewrite Problem \eqref{eq:prob} under the form \eqref{eq:breg}, so that it becomes equivalent to the computation of the proximity operator of $g$
with respect to the Bregman divergence given by

\begin{equation}
(\forall \vC \in \sm) \quad  D^f(\vC,\vY) =
\begin{cases}
\displaystyle\log\Big(\frac{\det(\vY)}{\det(\vC)}\Big) + \tr{\vY^{-1}\vC}-n & \mbox{if $\vC \in\smp$}\\
+ \infty & \mbox{otherwise.}
\end{cases}
\end{equation}

 In Table \ref{tab:logdet}, we list some particular choices for $g_0$, 
 and provide the associated closed form expression of the proximity operator $\prox_{\gamma (\varphi+ \psi)}$ for $\gamma \in ]0,+\infty[$, where $\varphi$ is defined in \eqref{eq:logdetfuncspec}. These expressions were derived from \cite[Table 10.2]{combettes:hal-00643807}.

\begin{rem}
Let $g_0$ be any of the convex spectral functions listed in Table \ref{tab:logdet}.
Let $\mathbf{W}$ be an invertible matrix in $\mathbb{R}^{n\times n}$, and let $\overline{\vC} \in \sm$
From the above results,  one can deduce the minimizer of 
$\vC \mapsto \gamma(f(\vC)+g_0(\mathbf{W}\vC \mathbf{W}^\top))+\frac12\| \mathbf{W}\vC \mathbf{W}^\top - \overline{\vC}\|_{\rm F}^2$
 where $\gamma \in ]0,+\infty[$. Indeed, by making a change of variable and by using basic properties of the $\log \det$ function,
this minimizer is equal to $\mathbf{W}^{-1}\prox_{\gamma(f+g_0)}(\overline{\vC}) (\mathbf{W}^{-1})^\top$.
\end{rem}

\begin{table}[h]
\caption{Proximity operators of $\gamma(f +g_0)$ with $\gamma > 0$ and $f$ given by \eqref{eq:logdetfunc}, evaluated at a symmetric matrix with vector of eigenvalues 
$\bm{\lambda}=(\lambda_i)_{1\le i \le n}$. For the inverse Schatten penalty, the function is set to $+\infty$ when the argument
$\vC$ is not positive definite.
$E_2$ denotes the set of matrices in $\sm$ with eigenvalues between $\alpha$
and $\beta$.
In the last line, the $i$-th component of the proximity operator is obtained by searching among the positive roots of a fourth order polynomial
those minimizing $\lambda'_i \mapsto \frac12 (\lambda'_i - \lambda_i )^2+\gamma\big(\mu \log((\lambda'_i)^2+\varepsilon)-\log \lambda'_i\big)$.
}
\label{tab:logdet}
\resizebox{\columnwidth}{!}{
\begin{tabular}{c|c}
\toprule	
$g_0(\vC)$, $\mu>0$	
& $\prox_{\gamma(\varphi + \psi)}(\bm{\lambda})$		\\

\toprule	
\toprule
Nuclear norm 				
&\multicolumn{1}{c}{\multirow{2}{*}{$\frac12 \left(\lambda_i-\gamma\mu +\sqrt{(\lambda_i-\gamma\mu )^2+4\gamma }\right)_{1\le i \le n}$}} \\
$\mu \mathcal{R}_1(\vC)$ 				
& \\
\midrule
Squared Frobenius norm 			
& 	\multicolumn{1}{c}{\multirow{2}{*}{$\displaystyle\frac{1}{2(2\gamma\mu+1)}\Big(\lambda_i+\sqrt{\lambda_i^2+4\gamma(2\gamma\mu+1)}\Big)_{1\le i \le n}$}	}			\\
$\mu \|\vC\|_{\rm F}^2$ 
& \\
\midrule
Schatten $p$--penalty			
& $\big(d_i\big)_{1\le i \le n}$\\ 
$\mu \mathcal{R}_p^p(\vC)$, $p \geq 1$	
& with $(\forall i \in \{1,\ldots,n\})$ $d_i > 0$ and $\mu\gamma p d_i^p +d_i^2-\lambda_i d_i=\gamma $ \\
\midrule
Inverse Schatten $p$--penalty			
& $\big(d_i\big)_{1\le i \le n}$\\ 
$\mu \mathcal{R}_p^p(\vC^{-1})$, $p > 0$	
& with $(\forall i \in \{1,\ldots,n\})$ $d_i > 0$ and $ d_i^{p+2} -\lambda_i d_i^{p+1}-\gamma d_i^p=\mu\gamma p $ \\
\midrule
Bounds on eigenvalues	
& \multicolumn{1}{c}{\multirow{2}{*}{$\lp\min\!\Big(\!\max\Big(\frac12 \big( \lambda_i + \sqrt{\lambda_i^2 +4\gamma}\big),\alpha\Big),\beta\Big)\rp_{1\le i \le n}$, $[\alpha,\beta] \subset [0,+\infty]$}}\\
$\iota_{E_2}(\vC)$  & \\
\midrule\midrule

 Cauchy 				
 
 & $\in \big\{(d_i)_{1\le i \le n}\mid (\forall i \in \{1,\ldots,n\})\; d_i > 0$ and\\
$\mu \log\det(\vC^2+\varepsilon \Id)$, $\varepsilon > 0$	
& $\qquad\qquad d_i^4-\lambda d_i^3 + \big(\varepsilon +\gamma(2\mu-1)\big)d_i^2 -\varepsilon\lambda_i d_i =\gamma\varepsilon \big\}$\\
\bottomrule
\end{tabular}
}
\end{table}

\subsubsection{Von Neumann Entropy}
\label{subsub:neu}
Our third example is the negative Von Neumann entropy, which appears to be useful in some
quantum mechanics problems \cite{bengtsson2006}. It is defined as
\begin{equation}
\label{eq:neu}
(\forall \vC \in \sm) \quad 
f(\vC) = 
\begin{cases}
\tr{\vC\log(\vC)} & \mbox{if $\vC \in\smsp$}\\
+\infty & \mbox{otherwise.}
\end{cases}
\end{equation} 
In the above expression, if $\vC = \vU \Diag(\bm{\lambda})\vU^\top$ with $\bm{\lambda}=(\lambda_i)_{1\le i \le n}\in  ]0,+\infty[^n$
and $\vU \in \mathcal{O}_n$, then  $\log(\vC) = \vU \Diag\big((\log\lambda_i)_{1\le i \le n}\big) \vU^\top$.
The logarithm of  a symmetric definite positive matrix is uniquely defined and
the function $\vC \mapsto \vC\log(\vC)$ can  be extended by continuity on $\smsp$ similarly to the case when $n=1$.
Thus, $f$ is the spectral function associated with

\begin{equation}
\label{eq:neuspec}
\big(\forall \bm{\lambda} = (\lambda_i)_{1\le i \le n} \in \mathbb{R}^n\big) \qquad 
\varphi(\bm{\lambda}) = 
\begin{cases}
\displaystyle\sum_{i=1}^n \lambda_i\log(\lambda_i) & \mbox{if }\bm{\lambda} \in [0,+\infty[^n\\
+\infty & \mbox{otherwise.}
\end{cases}
\end{equation}
Note that the Von Neumann entropy defined for symmetric matrices is simply equal to the well--known Shannon entropy \cite{cover2006elements} of the input eigenvalues. With this choice for function $f$, 
by setting $\vT = \log(\vY) + \Id$ where $\vY \in \smp$, Problem \eqref{eq:prob} can be recast under the form \eqref{eq:breg}, so that it becomes equivalent to the computation of the proximity operator of $g$ with respect to the Bregman divergence associated with the Von Neumann entropy:
\begin{multline*}
(\forall \vC \in \sm) \quad  D^f(\vC,\vY)  = \\
\begin{cases}
\tr{\vC\log(\vC)-\vY\log(\vY)-\lp\log(\vY) + \Id\rp\lp \vC-\vY\rp} & \mbox{if $\vC \in \smsp$}\\
+ \infty & \mbox{otherwise}.
\end{cases}
\end{multline*}

We provide in Table \ref{tab:vann} a list of closed form expressions of the proximity operator of $\gamma(f + g_0)$ for several choices of the spectral function $g_0$. 

\begin{table}[h]
\caption{Proximity operators of $\gamma(f +g_0)$ with $\gamma > 0$ and $f$ given by \eqref{eq:neu}, evaluated at a symmetric matrix with vector of eigenvalues 
$\bm{\lambda}=(\lambda_i)_{1\le i \le n}$. $E_2$ denotes the set of matrices in $\sm$ with eigenvalues between $\alpha$
and $\beta$. $\rm{W}(\cdot)$ denotes the W-Lambert function \cite{Corless1996}.}
\label{tab:vann}
\resizebox{\columnwidth}{!}{
\begin{tabular}{c|c}
\toprule
$g_0(\vC)$, $\mu > 0$ 				
& $\prox_{\gamma(\varphi+\psi)}(\bm{\lambda})$ 		\\
\toprule
\toprule
Nuclear norm 				
& \multicolumn{1}{c}{\multirow{2}{*}{$\gamma\lp{\rm W} \lp \frac{1}{\gamma}\exp\lp \frac{\lambda_i}{\gamma}-\mu-1\rp\rp\rp_{1\le i \le n}$}}\\
$\mu \mathcal{R}_1(\vC)$ 				
&  \\
\midrule
Squared Frobenius norm 			
&\multicolumn{1}{c}{\multirow{2}{*}{$\frac{\gamma}{2\mu\gamma+1}\lp{\rm W} \lp \frac{2\mu\gamma+1}{\gamma}\exp\lp \frac{\lambda_i}{\gamma}-1\rp\rp\rp_{1\le i \le n}$}} \\
$\mu \|\vC\|^2_{\rm F}$	
&  \\
\midrule
Schatten $p$--penalty 			
& $\big(d_i\big)_{1\le i \le n}$\\ 
$\mu \mathcal{R}_p^p(\vC)$, $p \geq 1$	
& with $(\forall i \in \{1,\ldots,n\})$ $d_i > 0$ and $p\mu\gamma d_i^{p-1} + d_i +\gamma\log d_i + \gamma = \lambda_i$ \\					
\midrule
Bounds on eigenvalues	
& \multicolumn{1}{c}{\multirow{2}{*}{$\lp\min\lp\max\lp\gamma{\rm W}\Big(\frac{1}{\gamma}\exp\Big(\frac{\lambda_i}{\gamma}-1\Big)\Big),\alpha\rp,\beta\rp\rp_{1\le i \le n}$, $[\alpha,\beta] \subset [0,+\infty]$}}\\
$\iota_{E_2}(\vC)$  & \\
\midrule\midrule
Rank				
& $(d_i)_{1\le i \le n}$ with\\
$\mu \operatorname{rank}(\vC)$ 				
&  $(\forall i \in \{1,\ldots,n\})\;\;d_i=\begin{cases} \rho_i & \mbox{if $\rho_i > \chi$}\\ 
0\,\text{or}\,\rho_i & \mbox{if $\rho_i = \chi$}\\ 0 & \text{otherwise}\end{cases}\mbox{ and} \begin{cases}\chi = \sqrt{\gamma(\gamma+2 \mu)}-\gamma, 
\\ \rho_i = \gamma {\rm W} \lp \frac{1}{\gamma}\exp\lp \frac{\lambda_i}{\gamma}-1\rp \rp\end{cases} $\\\\
\bottomrule
\end{tabular}
}
\end{table}

\subsection{Douglas-Rachford Algorithm}
\label{sub:DRA}
We now propose a Douglas-Rachford (DR) approach (\cite{doi:10.1137/0716071,combettes:hal-00643807,combettes:hal-00621820}) for numerically solving Problem \eqref{eq:prob}. The DR 
method minimizes the sum of $f -\tr{\vT \cdot} + g_0$ and $g_1$ by alternately computing proximity operators of each of these functions. 
Proposition \ref{prop:proxDR} allows us to calculate the proximity operator of $\gamma(f -\tr{\vT \cdot} + g_0)$ with $\gamma \in ]0,+\infty[$, by possibly using the expressions 
listed in Tables \ref{tab:froeb}, \ref{tab:logdet}, and \ref{tab:vann}. Since $g_1$ is not a spectral function, $\prox_{\gamma g_1}$ has to be derived from other expressions
of proximity operators. For instance, if $g_1$ is a separable sum of functions of its elements, e.g. $g = \| \cdot\|_1$, standard expressions for the proximity operator of vector functions
can be employed  \cite{0266-5611-23-4-008,combettes:hal-00643807}.\footnote{See also \textsf{http://proximity-operator.net}.}

\begin{algorithm}[h]
\caption{Douglas--Rachford Algorithm for solving Problem \eqref{eq:prob}}\label{al:DR}
\begin{algorithmic}[1]
\STATE Let  $\vT$ be a given matrix in $\sm$, set $\gamma>0$ and $\vC^{(0)}\in\sm$.
\FOR{$k=0,1,\dots$}
\STATE Diagonalize $\vC^{(k)}+\gamma \vT$, i.e. find $\vU^{(k)} \in \mathcal{O}_n$ and $\bm{\lambda}^{(k)} \in \mathbb{R}^n$ such that
$$\vC^{(k)}+\gamma \vT=\vU^{(k)}\Diag(\bm{\lambda}^{(k)})(\vU^{(k)})^\top$$
\STATE $\vd^{(k+\frac12)} \in \operatorname{Prox}_{\gamma \lp \varphi + \psi\rp}\left( \bm{\lambda}^{(k)}\right)$ 
\STATE $\vC^{(k+\frac12)} = \vU^{(k)}\Diag(\vd^{(k+\frac12)})(\vU^{(k)})^\top$
\STATE Choose $\alpha^{(k)}\in [0,2]$
\STATE $ \vC^{(k+1)} \in \vC^{(k)} + \alpha^{(k)} \Big(\operatorname{Prox}_{\gamma g_1}(2\vC^{(k+\frac12)}-\vC^{(k)})-\vC^{(k+\frac12)}\Big)$. 
\ENDFOR
\end{algorithmic}
\end{algorithm}
The computations to be performed are  summarized in Algorithm \ref{al:DR}. We state a convergence theorem in the matrix framework, which is an 
offspring of  existing results in arbitrary Hilbert spaces (see, for example, \cite{combettes:hal-00643807} and \cite[Proposition 3.5]{pesquet:hal-00790702}).
\begin{theorem}
\label{thh:DR}
Let $f$ and $g_0$ be functions satisfying  \eqref{eq:fphi} and \eqref{eq:g0}, respectively, where $\varphi\in \Gamma_0(\mathbb{R}^n)$
 and $\psi\in \Gamma_0(\mathbb{R}^n)$. Let $g_1 \in \Gamma_0(\sm)$ be such that $f-\tr{\vT \cdot}+g_0+g_1$ is coercive. Assume that 
 the intersection of the relative interiors of the domains of $f+g_0$ and $g_1$ is non empty. 
 Let $(\alpha^{(k)})_{k\ge 0}$ be a sequence in $[0,2]$ such that $\sum_{k=0}^{+\infty} \alpha^{(k)}(2-\alpha^{(k)}) = +\infty$.
 Then, the sequences    $(\vC^{(k+\frac12)})_{k \geq 0}$ and $\big(\operatorname{prox}_{\gamma g_1}(2\vC^{(k+\frac12)}-\vC^{(k)})\big)_{k \geq 0}$ generated by Algorithm~\ref{al:DR} converge to a solution to Problem~\eqref{eq:prob} where $g=g_0+g_1$.
\end{theorem}
We have restricted the above  convergence analysis to the convex case. Note however that recent convergence results for the DR algorithm in a non-convex setting are available in \cite{Arago2013,Li2016} for specific choices of the involved functionals.
 
\subsection{Positive Semi-Definite Constraint}
Instead of solving Problem~\eqref{eq:prob},  one may be interested in:
\begin{equation}
\label{e:probdefpos}
\minimize{\vC\in\mathcal{S}_n^+}{f(\vC) -\tr{\vC\vT} + g(\vC)},
\end{equation}
when $\dom f\cap \dom g \not \subset \mathcal{S}_n^+$.
This problem can be recast as minimizing over $\mathcal{S}_n$
$f-\tr{\cdot\vT}+ \widetilde{g}_0 + g_1$ where $\widetilde{g}_0 = g_0 + \iota_{\mathcal{S}_n^+}$.
We are thus coming back to the original formulation where $\widetilde{g}_0$ has been substituted for $g_0$.
In order to solve this problem with the proposed proximal approach, a useful result is stated below. 
\begin{proposition}
\label{prop:proxsemidefpos}
Let $\gamma \in ]0,+\infty[$ and $\overline{\vC}  \in \sm$.
 Let $f$ and $g_0$ be functions satisfying  \eqref{eq:fphi} and \eqref{eq:g0}, respectively, where $\varphi\in \Gamma_0(\mathbb{R}^n)$
 and $\psi\in \Gamma_0(\mathbb{R}^n)$.
Assume that 
 \begin{equation}
 \big(\forall \bm{\lambda}' = (\lambda'_i)_{1\le i \le n} \in \mathbb{R}^n\big)\quad
 \varphi(\bm{\lambda}')+\psi(\bm{\lambda}') = \sum_{i=1}^n \rho_i(\lambda'_i)
 \end{equation}
 where, for every $i\in \{1,\ldots,n\}$, $\rho_i\colon \mathbb{R} \to ]-\infty,+\infty]$ is such that $\dom \rho_i \cap [0,+\infty[ \neq \varnothing$.
Let $\bm{\lambda} = (\lambda_i)_{1\le i \le n} \in \mathbb{R}^n$ and $\vU\in\mathcal{O}_n$ be such that $\overline{\vC}+\gamma \vT = \vU\Diag(\bm{\lambda}) \vU^\top$.
Then
\begin{equation}
\label{e:proxSDP}
\prox_{\gamma \lp f  - \tr{\vT\, \cdot} +\widetilde{g}_0\rp}(\overline{\vC})=
 \vU \Diag\lp\big(\max(0, \prox_{\gamma \rho_i}(\lambda_i))\big)_{1\le i \le n}\rp\vU^\top.
\end{equation}
 \end{proposition}
\begin{proof}
Expression \eqref{e:proxSDP} readily follows from \autoref{prop:proxDR}(ii) and \cite[Proposition 2.2]{NellyPois}.\qed
\end{proof}


%
\section{Application to Covariance Matrix Estimation}
\label{sec:MMDR}
Estimating the covariance matrix of a random vector is a key problem in statistics, signal processing over graphs, and machine learning. Nonetheless, in existing optimization techniques, little attention is usually paid to the presence of noise corrupting the available observations. We show in this section how the results obtained in the previous sections can be used to tackle this problem in various contexts.
\subsection{Model and Proposed Approaches}
\label{sub:model}
Let $\vS \in \smsp$ be a sample estimate of a covariance matrix $\boldsymbol{\Sigma}$ which is assumed to
be decomposed as
\begin{equation}
\label{eq:xcor}
\boldsymbol{\Sigma}  = \vY^* + \sigma^2 \Id
\end{equation}
where $\sigma \in [0,+\infty[$ and $\vY^*\in \smsp$ may have a low-rank structure.
Our objective in this section will be to propose variational methods to provide an estimate of $\vY^*$ from $\vS$
by assuming that $\sigma$ is known.
Such a problem arises when considering the following observation model \cite{sun17}:
\begin{equation}
\label{eq:palomar}
(\forall i \in \lg1,\ldots,N\rg) \quad \iter{\vx}{}{(i)} = \vA\iter{\vs}{}{(i)} +\iter{\ve}{}{(i)}
\end{equation}
where $\vA\in\rd^{n\times m}$ with $m\le n$ and, for every $i \in \left\{1,\ldots,N\right\}$, $\iter{\vs}{}{(i)} \in \rd^m$ and $\iter{\ve}{}{(i)} \in \rd^n$ are realizations of mutually 
independent identically distributed Gaussian multivalued random variables with zero mean and covariance matrices $\vP\in \mathcal{S}_m ^{++}$ and $\sigma^2\Id$, respectively. This model has been employed for instance in \cite{Tipping01,wipf04} in the context of the ``Relevant Vector Machine problem''. The covariance matrix $\boldsymbol{\Sigma}$ of the  noisy input data $\left(\iter{\vx}{}{(i)}\right)_{1 \leq i \leq N}$ takes the 
form \eqref{eq:xcor} with $\vY^* = \vA\vP\vA^\top$.

On the other hand, a simple estimate of $\boldsymbol{\Sigma}$ from the observed data $\left(\iter{\vx}{}{(i)}\right)_{1 \leq i \leq N}$ is
\begin{equation}
\label{eq:covformula}
\vS = \frac{1}{N}\sum_{i=1}^N \vx^{(i)} \big(\vx^{(i)}\big)^\top.
\end{equation}
\paragraph*{Covariance-based model.\,} A first estimate $\widehat{\vY}$ of $\vY^*$ is given by
\begin{equation}
\label{eq:palF}
\widehat{\vY} =\argmin{\vY \in\smsp} \frac{1}{2}\|\vY- \vS+\sigma^2\Id\|_{\rm F}^2 + g_0(\vY) + g_1(\vY),
\end{equation}
where $\vS$ is the empirical covariance matrix, $g_0$ satisfies \eqref{eq:g0} with $\psi\in \Gamma_0(\mathbb{R}^n)$, $g_1 \in \Gamma_0(\sm)$, and
the intersection of the relative interiors of the  domains of $g_0$ and $g_1$ is assumed to be non empty.

A particular instance of this model with $\sigma=0$, $g_0=\mu_0 \mathcal{R}_1$, $g_1= \mu_1\|\cdot\|_1$, and $(\mu_0,\mu_1) \in [0,+\infty[^2$ was investigated in  \cite{Zhou14} and \cite{Richard:2012:ESS:3042573.3042584}  for estimating sparse low-rank  covariance matrices. In the latter reference, an application to real data processing arising from protein interaction and social network analysis is presented. 

One can observe that Problem \eqref{eq:palF} takes the form  \eqref{e:probdefpos}
by setting $f = \frac{1}{2}\|\cdot \|_{\rm F}^2 $ and $\vT = \vS-\sigma^2\Id$. This allows us to solve \eqref{eq:palF} with Algorithm \ref{al:DR}. Since it is assumed that $g_0$ satisfies \eqref{eq:g0},
the proximity step on $f+g_0+\iota_{\smsp}$ can be performed by employing \autoref{prop:proxsemidefpos} and formulas from \autoref{tab:froeb}. 
The resulting Douglas--Rachford procedure can thus be viewed as an alternative to the methods developed in 
\cite{Richard:2012:ESS:3042573.3042584} and 
\cite{Zhou14}. Let us emphasize that these two algorithms were devised to solve an instance of \eqref{eq:palF} 
corresponding to the aforementioned specific choices for $g_0$ and $g_1$, while our approach leaves more freedom in the choice of the regularization functions.  
\paragraph*{Precision-based model.\,}
An alternative strategy consists of focusing on the estimation of the inverse of the covariance matrix, i.e. the \emph{precision} matrix 
$\vC^* = (\vY^*)^{-1}$ by assuming that $\vY^* \in \smp$ but may have very small eigenvalues in order to model a possible low-rank structure.
Tackling the problem from this viewpoint leads us to propose the following penalized negative log-likelihood cost function: 
\begin{equation}
\label{eq:prec}
(\forall \vC \in \sm) \qquad \mathcal{F}(\vC) =f(\vC) +\mathcal{T}_\vS\lp\vC\rp  +g_0(\vC) + g_1(\vC)
\end{equation}
where
\begin{align}
(\forall \vC \in \sm) \qquad f(\vC) = \begin{cases}
\log\det\lp \vC^{-1} + \sigma^2\Id \rp  & \mbox{if $\vC\in\smp$}\\
+\infty & \mbox{otherwise,}
\end{cases}
\label{eq:logterm}\\
(\forall \vC \in \sm) \qquad \mathcal{T}_\vS(\vC) = \begin{cases}
\tr{\lp \Id + \sigma^2 \vC \rp^{-1}\vC\vS}  & \mbox{if $\vC\in\smsp$}\\
+\infty & \mbox{otherwise,}
\end{cases}
\label{eq:trterm}
\end{align}

$g_0\in\Gamma_0(\sm)$ satisfies \eqref{eq:g0} with $\psi\in \Gamma_0(\rd^n)$, and $g_1 \in \Gamma_0(\sm)$. 
Typical choices of interest for the latter two functions are
\begin{equation}
\label{eq:g0inv}
(\forall \vC \in \sm) \qquad g_0(\vC) = \begin{cases}
\mu_0 \mathcal{R}_1(\vC^{-1}) & \mbox{if $\vC\in\smp$}\\
+\infty & \mbox{otherwise,}
\end{cases}
\end{equation}
and $g_1 = \mu_1 \|\cdot\|_1$ with $(\mu_0,\mu_1) \in [0,+\infty[^2$. The first function serves to promote a desired low-rank property by penalizing small eigenvalues of the precision matrix, whereas the second one enforces the sparsity of this matrix as it is usual in graph inference problems. This constitutes a main difference with respect to the covariance-based model which is more suitable to estimate sparse covariance matrices. Note that the standard graphical lasso framework \cite{Friedman07} is then recovered by setting $\sigma = 0$ and $\mu_0 = 0$.
The advantage of our formulation is that it allows us to consider more flexible variational models while accounting for the presence of noise corrupting the observed data.
The main difficulty however is that Algorithm \autoref{al:DR} cannot be directly applied to minimize $\mathcal{F}$. In \autoref{sub:analysis}, we will study in more details the properties of the cost function. This will allow us to derive a novel optimization algorithm making use of our previously developed Douglas-Rachford scheme for its inner steps

\subsection{Study of Objective Function $\mathcal{F}$}
\label{sub:analysis}
The following lemma will reveal useful in our subsequent analysis.
\begin{lemma}
\label{lem:h}
Let $\sigma \in ]0,+\infty[$.
Let $h\colon ]0,\sigma^{-2}[ \to \rd$ be a twice differentiable function
and let
\begin{equation}
u \colon [0,+\infty[ 
\to \rd
\colon \lambda \mapsto 
\frac{\lambda}{1+\sigma^2 \lambda}. 
\label{eq:u}
\end{equation}

 The composition $h\circ u$ is convex on $]0,+\infty[$ if and only if
\begin{equation}
(\forall \upsilon \in ]0,\sigma^{-2}[) \quad \ddot{h}(\upsilon)(1-\sigma^2 \upsilon)-2\sigma^2 \dot{h}(\upsilon)\geq 0,
\end{equation}
where $\dot{h}$ (resp. $\ddot{h}$) denotes the first (resp. second) derivative of $h$.
\end{lemma}
\begin{proof}
The result directly follows from the calculation of the second-order derivative of $h\circ u$.\qed
\end{proof}
Let us now note that $f$ is a spectral function fulfilling \eqref{eq:fphi} with 
\begin{equation}
\label{eq:spectralnewlog}
\big(\forall \bm{\lambda} = (\lambda_i)_{1\le i \le n} \in \mathbb{R}^n\big) \qquad 
\varphi(\bm{\lambda}) = 
\begin{cases}
-\displaystyle\sum_{i=1}^n \log\big(u(\lambda_i)\big)& \mbox{if }\bm{\lambda} \in ]0,+\infty[^n\\
+\infty & \mbox{otherwise,}
\end{cases}
\end{equation}
where $u$ is defined by \eqref{eq:u}. 
According to \autoref{lem:h} (with $h = -\log$), $f \in \Gamma_0(\sm)$.
Thus, the assumptions made on $g_0$ and $g_1$, 
allow us to deduce that $f + g_0+g_1$ is convex and lower-semicontinuous on $\sm$. 

Let us now focus on the properties of the second term in \eqref{eq:prec}.

\begin{lemma}
Let $\vS\in\smsp$. The function $\mathcal{T}_\vS$ in \eqref{eq:trterm} is concave on $\smsp$.
\label{lemma:trace}
\end{lemma}
\begin{proof} 
By using differential calculus rules in \cite{magnus99}, we will show that the Hessian of $-\mathcal{T}_\vS$ 
evaluated at any matrix in $\smp$
is a positive semidefinite operator.
In order to lighten our notation,  for every invertible matrix $\vC$, let us define  $\vM = \vC^{-1} + \sigma^2\Id $.
Then, the first-order differential of $\mathcal{T}_\vS$ at every $\vC\in \smp$ is
\begin{eqnarray}
\diff\tr{\mathcal{T}_\vS(\vC)}&
=& \tr{\lp\diff \vM^{-1}\rp\vS}\nonumber\\
&=& \tr{-\vM^{-1}(\diff \vM)\vM^{-1}\vS}\nonumber\\
&=& \tr{\lp \vC^{-1} + \sigma^2\Id\rp^{-1}\vS\lp \vC^{-1} + \sigma^2\Id\rp^{-1}\vC^{-1}(\diff \vC)\vC^{-1}} \nonumber\\
&=& \tr{\lp\Id +\sigma^2\vC\rp^{-1}\vS\lp\Id +\sigma^2\vC\rp^{-1}(\diff \vC)}.
\label{eq:gradTsproof}
\end{eqnarray}
We have used the expression of the differential of the inverse
\cite[Chapter~8, Theorem~3]{magnus99} and 

the invariance of the trace with respect to cyclic permutations. 
It follows from \eqref{eq:gradTsproof} that the gradient of $\mathcal{T}_\vS$ reads
\begin{equation}
\label{eq:1der}
(\forall \vC \in \smp) \quad \nabla \mathcal{T}_\vS(\vC) = \lp\Id +\sigma^2\vC\rp^{-1}\vS\lp\Id +\sigma^2\vC\rp^{-1}.
\end{equation}
In order to calculate the Hessian $\mathfrak{H}$ of $\mathcal{T}_\vS$, 
we calculate the differential of $\nabla\mathcal{T}_\vS$.
 Again, in order to simplify our notation, for every matrix $\vC$, we define 
 \begin{equation}
\vN = \Id+\sigma^2\vC \quad \Rightarrow \quad \diff \vN = \sigma^2\diff \vC.
\end{equation}
The differential of $\nabla\mathcal{T}_\vS$ at every $\vC\in \smp$ then reads
\begin{eqnarray}
\diff\vect\lp\nabla \mathcal{T}_\vS(\vC)\rp 
&=& 
\vect\lp \diff (\vN^{-1}\vS\vN^{-1})\rp \nonumber\\
&=& \vect\lp (\diff \vN^{-1})\vS\vN^{-1}+ \vN^{-1}(\diff \vS\vN^{-1})\rp \nonumber\\
&=& -\vect(\vN^{-1}(\diff \vN) \vN^{-1}\vS \vN^{-1}) -\vect\lp \vN^{-1}\vS\vN^{-1}(\diff \vN) \vN^{-1} \rp \nonumber\\
&=& -\lp \lp \vN^{-1}\vS \vN^{-1}\rp^\top\otimes \vN^{-1}\rp\vect(\diff \vN) - \lp \lp \vN^{-1} \rp^\top\otimes \vN^{-1}\vS \vN^{-1} \rp \vect(\diff \vN)\nonumber \\
&=& - \big( \lp \vN^{-1}\vS \vN^{-1}\rp\otimes \vN^{-1} + \vN^{-1} \otimes \lp \vN^{-1}\vS \vN^{-1}\rp  \big) \diff\vect( \vN)\nonumber\\
&= & \mathfrak{H}(\vC) \diff\vect( \vC)\nonumber
\end{eqnarray}
with
\begin{equation}
\label{eq:h}
\mathfrak{H}(\vC)  
= -\sigma^2\lp\nabla\mathcal{T}_\vS\lp\vC\rp\otimes \lp\Id+\sigma^2\vC\rp^{-1} +  \lp\Id+\sigma^2\vC\rp^{-1}\otimes\nabla\mathcal{T}_\vS\lp\vC\rp\rp.
\end{equation}
To derive the above expression, we have used the facts that, for every $\vA \in \mathbb{R}^{n\times m}$, $\vX \in \mathbb{R}^{m \times p}$, and $\vB \in \mathbb{R}^{p \times q}$, $\vect\lp\vA\vX\vB\rp = \lp\vB^\top \otimes \vA \rp\vect \vX$ \cite[Chapter~2,Theorem~2]{magnus99}
and that matrices $\vN$ and $\vS$ are symmetric.

Let us now check that, for every $\vC \in \smp$, $\mathfrak{H}(\vC)$ is negative semidefinite.
It follows from expression \eqref{eq:1der}, the symmetry of $\vC$, and the positive semidefiniteness of $\vS$ that $\nabla\mathcal{T}_\vS(\vC)$ belongs to $\smsp$.
Since 
\begin{align*}
\big(\nabla\mathcal{T}_\vS\lp\vC\rp\otimes \lp\Id+\sigma^2\vC\rp^{-1}\big)^\top &= \big(\nabla\mathcal{T}_\vS\lp\vC\rp\big)^\top \otimes \big(\lp\Id+\sigma^2\vC\rp^{-1}\big)^\top
\\
&= \nabla\mathcal{T}_\vS\lp\vC\rp\otimes \lp\Id+\sigma^2\vC\rp^{-1}, 
\end{align*}
$\nabla\mathcal{T}_\vS\lp\vC\rp\otimes \lp\Id+\sigma^2\vC\rp^{-1}$ is symmetric.
Let us denote by $(\gamma_i)_{1\le i \le n} \in [0,+\infty[^n$ the eigenvalues of $\nabla\mathcal{T}_\vS\lp\vC\rp$ and 
by $(\zeta_i)_{1\le i \le n} \in [0,+\infty[^n$ those of $\vC$. According to \cite[Chapter 2, Theorem 1]{magnus99}, the eigenvalues of
$\nabla\mathcal{T}_\vS\lp\vC\rp\otimes \lp\Id+\sigma^2\vC\rp^{-1}$ are $\big(\gamma_i/(1+\sigma^2 \zeta_j)\big)_{1\le i,j \le n}$ and they are therefore nonnegative.
This allows us to claim that $\nabla\mathcal{T}_\vS\lp\vC\rp\otimes \lp\Id+\sigma^2\vC\rp^{-1}$ belongs to $\mathcal{S}_{n^2}^+$.
For similar reasons, $\lp\Id+\sigma^2\vC\rp^{-1}\otimes\nabla\mathcal{T}_\vS\lp\vC\rp \in \mathcal{S}_{n^2}^+$, which allows us to conclude that $-\mathfrak{H}(\vC)\in \mathcal{S}_{n^2}^+$.
Hence, we have proved that $\mathcal{T}_\vS$ is concave on $\smp$. By continuity of $\mathcal{T}_\vS$ relative to $\smsp$, the concavity property extends on $\smsp$.\qed

\end{proof}
As a last worth mentioning property, $\mathcal{T}_\vS$ is bounded on $\smp$. So, if $\dom f \cap \dom g_0 \cap \dom g_1 \neq \varnothing$ and $f+g_0+g_1$ is coercive, then 
there exists a minimizer of $\mathcal{F}$. Because of the form of $f$, the coercivity condition is satisfied if $g_0+g_1$ is lower bounded and 
$\lim_{\vC \in \smsp, \| \vC \| \to +\infty} g_0(\vC)+g_1(\vC) = +\infty$.

\subsection{Minimization Algorithm for $\mathcal{F}$}
\label{sub:algo}
In order to find a minimizer of $\mathcal{F}$, we propose a \emph{Majorize--Minimize} (MM) approach, following the ideas in \cite{chouzenoux16,sun17,doi:10.1198/0003130042836,Jacobson07}. At each iteration of an MM algorithm, one constructs a tangent function that majorizes the given cost function and is equal to it at the current iterate. The next iterate is obtained by minimizing this tangent majorant function, resulting in a sequence of iterates that reduces the cost function value monotonically.
According to the results stated in the previous section, our objective function reads as a difference of convex terms. We  propose to build a majorizing approximation of function $\mathcal{T}_\vS$ at $\vC' \in \smp$ by exploiting \autoref{lemma:trace} and the classical concavity inequality on $\mathcal{T}_\vS$ :
\begin{equation}
\label{eq:linear_app}
(\forall \vC \in \smp) \quad \mathcal{T}_\vS\lp\vC\rp\leq \mathcal{T}_\vS\lp\vC'\rp + \tr{\nabla\mathcal{T}_\vS(\vC')\lp \vC-\vC'\rp}.
\end{equation}
As $f$ is finite only on $\smp$, a tangent majorant of the cost function \eqref{eq:prec} at $\vC'$ reads:
\[
(\forall \vC \in \sm) \quad \mathcal{G}(\vC \mid \vC') = f\lp\vC\rp + \mathcal{T}_\vS\lp\vC'\rp + \tr{\nabla\mathcal{T}_\vS(\vC')\lp \vC-\vC'\rp} + g_0(\vC) + g_1(\vC).
\]
This leads to the general MM scheme:
\begin{equation}
\label{eq:majo}
(\forall \ell \in \mathbb{N}) \quad \vC^{(\ell+1)}\in \Argmin{\vC\in\sm}f(\vC) + \operatorname{trace}\big(\nabla\mathcal{T}_\vS(\vC^{(\ell)})\vC\big) +g_0(\vC) + g_1(\vC)
\end{equation}
with $\vC^{(0)} \in \smp$. At each iteration of the MM algorithm, we have then to solve a convex optimization problem of the form \eqref{eq:prob}.
In the case when $g_1\equiv 0$, we can employ the procedure described in \autoref{sec:Dec} to perform this task in a direct manner. 
The presence of a regularization term $g_1\not \equiv 0$ usually prevents us to have an explicit solution to the inner minimization problem involved in the MM procedure. We then propose 
in Algorithm \ref{al:MMDR} to resort to the Douglas--Rachford approach in \autoref{sec:DR} to solve it iteratively.
\begin{algorithm}[H]
\caption{MM algorithm with DR inner steps}\label{al:MMDR}
\begin{algorithmic}[1]
\STATE Let $\vS\in\smsp$  be the data matrix. Let $\varphi$ be as in \eqref{eq:spectralnewlog}, let $\psi\in\Gamma_0(\rd^n)$ be associated with $g_0$. Let $(\gamma_\ell)_{\ell \in \mathbb{N}}$ be a sequence in $]0,+\infty[$.
Set $\vC^{(0,0)} = \vC^{(0)}\in\smp$.   
\FOR{$\ell=0,1,\dots$}
\FOR{$k=0,1,\dots$}
\STATE Compute $\vU^{(\ell,k)}\in \mathcal{O}_n$ and $\bm{\lambda}^{(\ell,k)}\in \mathbb{R}^n$ such that $$\vC^{(\ell,k)}-\gamma_\ell \nabla\mathcal{T}_\vS(\vC^{(\ell)})=\vU^{(\ell,k)}\Diag(\bm{\lambda}^{(\ell,k)})\lp \vU^{(\ell,k)}\rp^\top$$ \label{e:startDRalg2}
\STATE $ \vd^{\lp \ell,k+\frac{1}{2}\rp} = \textnormal{prox}_{\gamma_\ell \lp \varphi + \psi\rp}\left( \bm{\lambda}^{(\ell,k)}\right)$ 
\STATE $ \vC^{\lp \ell,k+\frac{1}{2}\rp} = \vU^{(\ell,k)}\Diag\lp\vd^{\lp \ell,k+\frac{1}{2}\rp}\rp\lp\vU^{(\ell,k)}\rp^\top$
\IF{Convergence of MM sub-iteration is reached}\label{s:stopalg2}
\STATE $\vC^{(\ell+1)} = \vC^{(\ell,k+\frac{1}{2})}$
\STATE $\vC^{(\ell+1,0)} = \vC^{(\ell,k)}$
\STATE \textbf{exit} inner loop
\ENDIF
\STATE  Choose $\alpha_{\ell,k}\in]0,2[$
\STATE $ \vC^{(\ell,k+1)} = \vC^{(\ell,k)} + \alpha_{\ell,k}\left( \textnormal{prox}_{\gamma_\ell g_1}\left(2\vC^{\lp \ell,k+\frac{1}{2}\rp}-\vC^{(\ell,k)}\right)-\vC^{(\ell,k+\frac{1}{2})}\right)$ 
\label{e:stopDRalg2}
\ENDFOR
\ENDFOR
\end{algorithmic}
\end{algorithm}

A convergence result is next stated, which is inspired from \cite{Wu_83} (itself relying on \cite[p.~6]{Zangwill_69}), but does not require the differentiability
of $g_0+g_1$.
\begin{theorem}
\label{thh:MMconv}
Let $(\vC^{(\ell)})_{\ell \ge 0}$ be a sequence generated by \eqref{eq:majo}. 
Assume that\linebreak $\dom f \cap \dom g_0 \cap \dom g_1 \neq \varnothing$, $f+g_0+g_1$ is coercive, and 
$E = \{\vC \in \sm\mid \mathcal{F}(\vC) \le \mathcal{F}(\vC^{(0)})\}$ is a subset of the relative interior of $\dom g_0 \cap \dom g_1$.
Then, the following properties hold:
\begin{enumerate}
\item[(i)]  $\big(\mathcal{F}(\vC^{(\ell)})\big)_{\ell \ge 0}$ is a decaying sequence converging
to $\widehat{\mathcal{F}}\in \mathbb{R}$.
\item[(ii)] $(\vC^{(\ell)})_{\ell \ge 0}$ has a cluster point. 
\item[(iii)] Every cluster point $\widehat{\vC}$ of  $(\vC^{(\ell)})_{\ell \ge 0}$
is such that $\mathcal{F}(\widehat{\vC}) = \widehat{\mathcal{F}}$ and it is a critical point of $\mathcal{F}$, i.e.
$-\nabla f(\widehat{\vC}) - \nabla\mathcal{T}_\vS(\widehat{\vC}) \in \partial (g_0+g_1)(\widehat{\vC})$.
\end{enumerate}
\end{theorem} 
\begin{proof}
First note that $(\vC^{(\ell)})_{\ell \ge 0}$ is properly defined by \eqref{eq:majo} since, for every $\vC\in \smp$, $\mathcal{G}(\cdot \mid \vC)$ is a coercive 
lower-semicontinuous function. It indeed majorizes $\mathcal{F}$ which is coercive, since $f+g_0+g_1$ has been assumed coercive.\\
(i) As a known property of MM strategies, $\big(\mathcal{F}(\vC^{(\ell)})\big)_{\ell \ge 0}$ is a decaying sequence \cite{Jacobson07}.
Under our assumptions, we have already seen that $\mathcal{F}$ has a minimizer. We deduce that $\big(\mathcal{F}(\vC^{(\ell)})\big)_{\ell \ge 0}$
is lower bounded, hence convergent.\\
(ii) Since $\big(\mathcal{F}(\vC^{(\ell)})\big)_{\ell \ge 0}$ is a decaying sequence, $(\forall \ell \ge 0)$ $\vC^{(\ell)} \in E$.
Since $\mathcal{F}$ is proper, lower-semicontinuous, and coercive, $E$ is a nonempty compact set
and  $(\vC^{(\ell)})_{\ell \ge 0}$ admits a cluster point in $E$.\\
(iii) If $\widehat{\vC}$  is a cluster point of $(\vC^{(\ell)})_{\ell \ge 0}$, then there exists a subsequence $(\vC^{(\ell_k)})_{k \ge 0}$ converging
to $\widehat{\vC}$. Since $E$ is a nonempty subset of the relative interior of $\dom g_0 \cap \dom g_1$ and $g_0+g_1\in \Gamma_0(\sm)$, $g_0+g_1$ is continuous relative to $E$
\cite[Corollary~8.41]{Bauschke:2017}. As $f+\mathcal{T}_\vS$ is continuous on $\dom f\cap \dom\mathcal{T}_\vS =\smp$, $\mathcal{F}$ is continuous relative to $E$. Hence,
$\widehat{\mathcal{F}} = \lim_{k\to +\infty}  \mathcal{F}(\vC^{(\ell_k)}) = \mathcal{F}(\widehat{\vC})$.
On the other hand, by similar arguments applied to sequence $(\vC^{(\ell_k+1)})_{k \ge 0}$, there exists a subsequence $(\vC^{(\ell_{k_q}+1)})_{q \ge 0}$ converging to some
$\widehat{\vC}' \in E$ such that $\widehat{\mathcal{F}} = \mathcal{F}(\widehat{\vC}')$. In addition, thanks to \eqref{eq:majo}, we have
\begin{equation}
(\forall \vC \in \sm) (\forall q\in \mathbb{N}) \quad  \mathcal{G}(\vC^{(\ell_{k_q}+1)}\mid \vC^{(\ell_{k_q})}) \le \mathcal{G}(\vC\mid \vC^{(\ell_{k_q})}).
\end{equation}
By continuity of $f$ and $\nabla\mathcal{T}_\vS$ on $\smp$ and by continuity of 
$g_0+g_1$ relative to $E$, 
\begin{equation}\label{eq:MMhCphC}
(\forall \vC \in \sm) \quad \mathcal{G}(\widehat{\vC}'\mid \widehat{\vC}) \le \mathcal{G}(\vC\mid \widehat{\vC}).
\end{equation}
Let us now suppose that $\widehat{\vC}$ is not a critical point of $\mathcal{F}$. Since
the subdifferential of $\mathcal{G}(\cdot \mid \widehat{\vC})$ at $\widehat{\vC}$ is
$\nabla f(\widehat{\vC})+\nabla\mathcal{T}_\vS(\widehat{\vC})+\partial (g_0+g_1)(\widehat{\vC})$ \cite[Corollary 16.48(ii)]{Bauschke:2017},
the null matrix does not belong to this subdifferential, which means that $\widehat{\vC}$ is not a minimizer
of $\mathcal{G}(\cdot \mid \widehat{\vC})$ \cite[Theorem 16.3]{Bauschke:2017}. It follows from \eqref{eq:MMhCphC} and standard MM properties 
that $\mathcal{F}(\widehat{\vC}') \le 
\mathcal{G}(\widehat{\vC}'\mid \widehat{\vC}) < \mathcal{G}(\widehat{\vC} \mid \widehat{\vC}) = \mathcal{F}(\widehat{\vC})$.
The resulting strict inequality contradicts the already established fact that $\mathcal{F}(\widehat{\vC}')=\mathcal{F}(\widehat{\vC})$.\qed
\end{proof}

\section{Numerical Experiments}
\label{sec:Num}
This section presents some numerical tests illustrating the validity of the  proposed algorithms. 
More specifically, in \autoref{sub:comp} the Douglas--Rachford (DR) approach of \autoref{sec:DR} is compared with other state--of--the--art algorithms previously mentioned, namely Incremental Proximal Descent (IPD) \cite{Richard:2012:ESS:3042573.3042584} and ADMM \cite{Zhou14}, on a problem of covariance matrix estimation. In \autoref{sub:noisy}, we present an application of the MM approach from \autoref{sec:MMDR} to a graphical lasso problem in the presence of noisy data.
All the experiments were conducted on a MacBook Pro equipped with an Intel Core i7 at 2.2 GHz, 16 Gb of RAM (DDR3 1600 MHz), and Matlab R2015b. 

\subsection{Application to Sparse Covariance Matrix Estimation}
\label{sub:comp}
We first consider the application of the DR algorithm from \autoref{sec:DR} to the sparse covariance matrix estimation problem introduced in \cite{Richard:2012:ESS:3042573.3042584}. The objective is to retrieve an estimate of a low rank covariance matrix $\vY^* \in \mathcal{S}_n^+$ from $N$ noisy realizations $(\vx^{(i)})_{1 \leq i \leq N}$ of a Gaussian multivalued random vector with zero mean and covariance matrix $\vY^* + \sigma^2\Id$, with $\sigma > 0$. As we have shown in \autoref{sub:model}, a solution to this problem can be obtained by solving the penalized least-squares problem \eqref{eq:palF}, where $\vS$ is the empirical covariance matrix defined in \eqref{eq:covformula}, and the regularization terms are $g_0 = \mu_0\mathcal{R}_1$ and $g_1=\mu_1\|\cdot\|_1$. We propose to compare the performance of the DR approach from \autoref{sub:DRA}, with the IPD algorithm \cite{Richard:2012:ESS:3042573.3042584} and the ADMM procedure \cite{Zhou14}, for solving this convex optimization problem.

The synthetic data are generated using a procedure similar to the one in \cite{Richard:2012:ESS:3042573.3042584}. A block-diagonal covariance matrix $\vY^*$ is considered, composed with $r$ blocks with dimensions $(r_j)_{1\le j \le r}$, so that $n = \sum_{j=1}^r r_j$. The $j$-th diagonal block of $\vY^*$ reads as a product $\va_j\va_j^\top$, where the components of $\va_j\in\rd^{r_j}$ are randomly drawn on $[-1,1]$. The number of observations $N$ is equal to $n$ and $\sigma = 0.1$. The three algorithms are initialized with $\vS+\Id$, and stopped as soon as a relative decrease criterion on the objective function is met, i.e. when $|\mathcal{F}_{k+1}-\mathcal{F}_k|/|\mathcal{F}_k|\leq \varepsilon$, $\varepsilon> 0$ being a given tolerance and $\mathcal{F}_k$ denoting the objective function value at iteration $k$. The maximum number of iterations is set to $2000$. The penalty parameters $\mu_1$ and $\mu_0$ are chosen in order to get a reliable estimation of the original covariance matrix. The gradient stepsize for IPD is set to $k^{-1}$. In Algorithm \autoref{al:DR}, $\alpha_k$ is set to 1.5. In ADMM, the initial Lagrange multiplier is set to a matrix with all entries equal to one, and the parameter of the proximal step is set to 1. 

\begin{figure}[htbp]
\begin{center}
\renewcommand*{\thesubfigure}{}
\captionsetup[subfigure]{labelformat=simple}
\subfloat{\label{fig:true100}\includegraphics[scale=0.2]{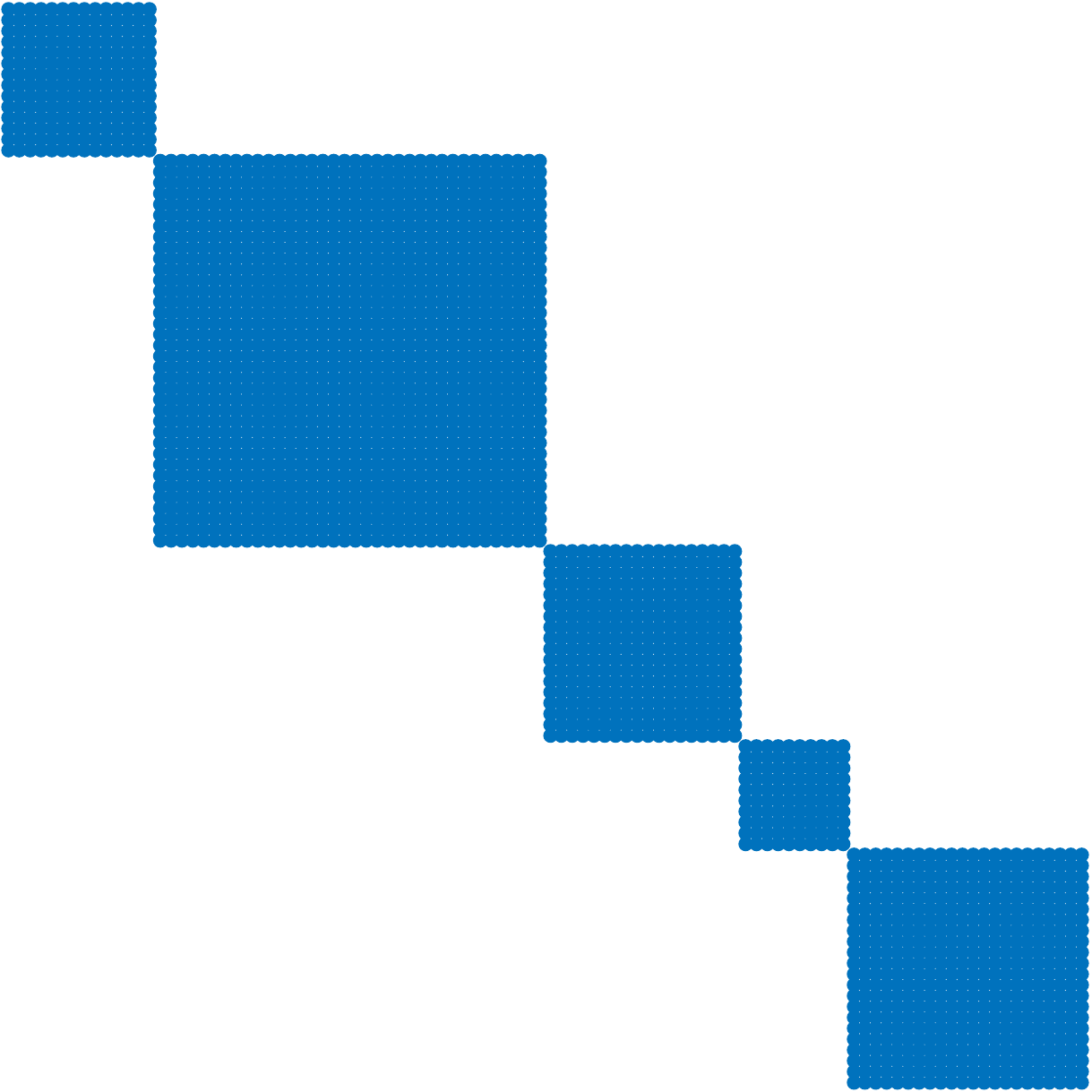}}\qquad \subfloat{\label{fig:DR100}\includegraphics[scale=0.2]{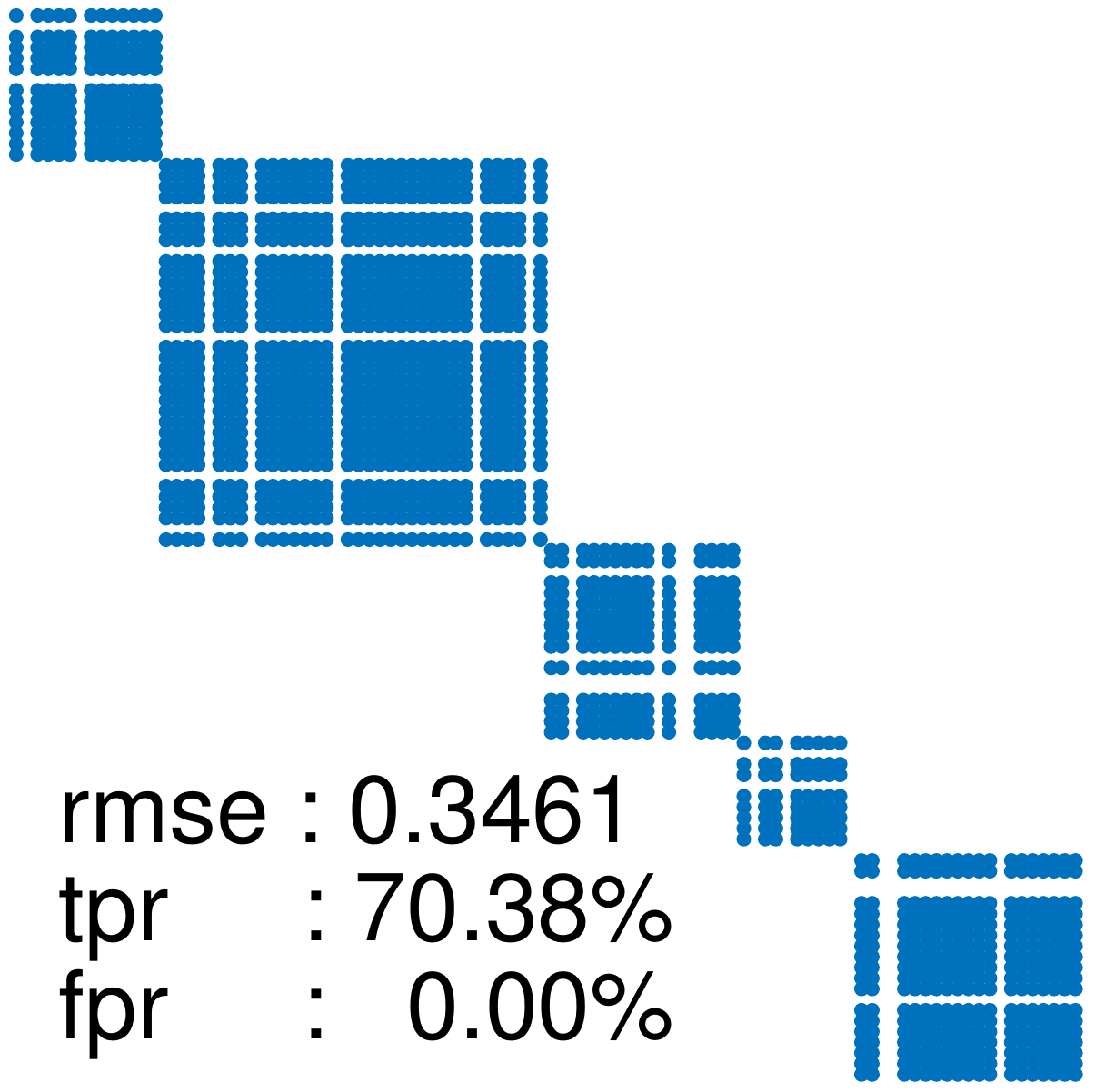}}\qquad \subfloat{\label{fig:ADMM100}\includegraphics[scale=0.2]{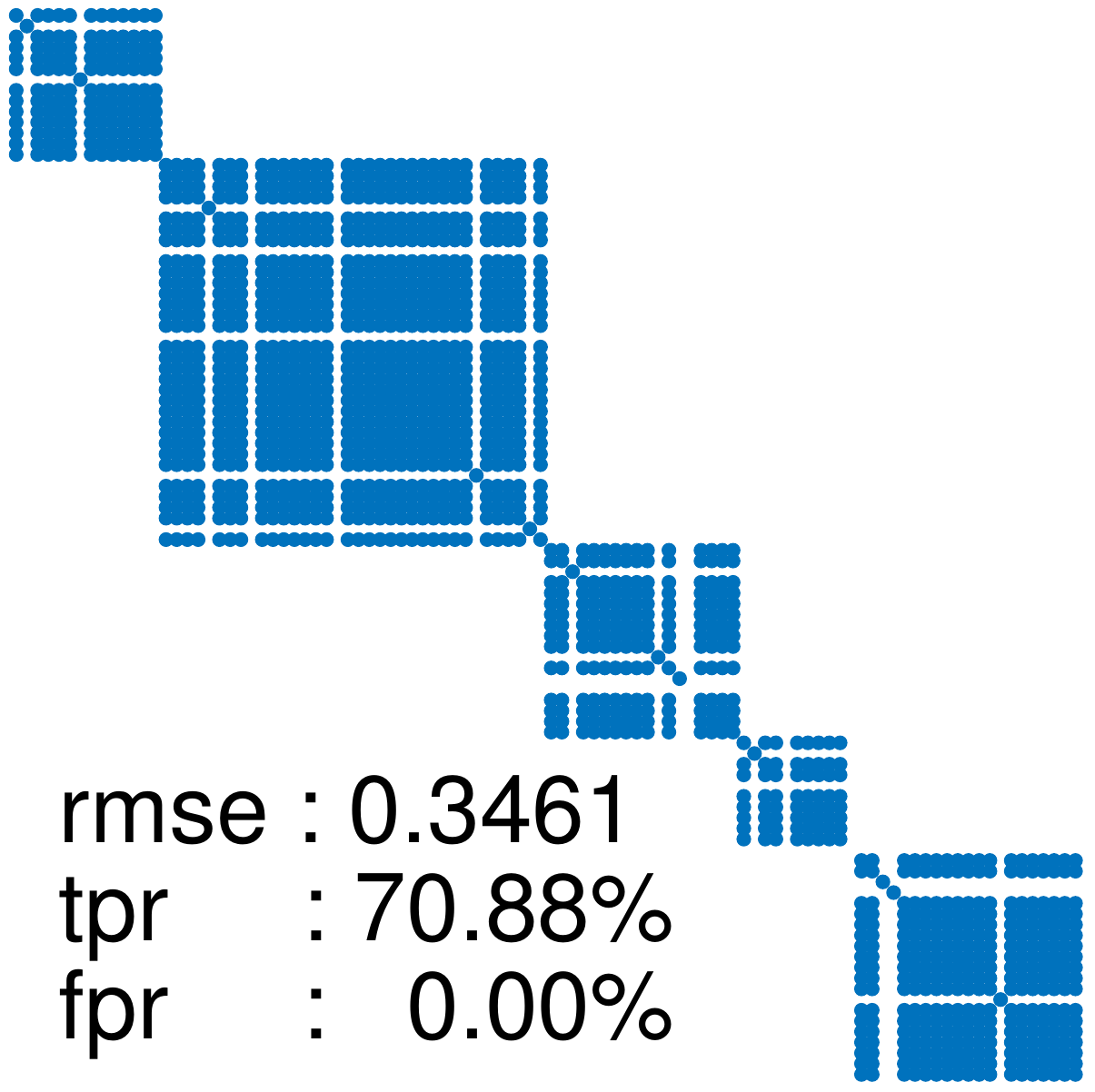}}\qquad \subfloat{\label{fig:IPD100}\includegraphics[scale=0.2]{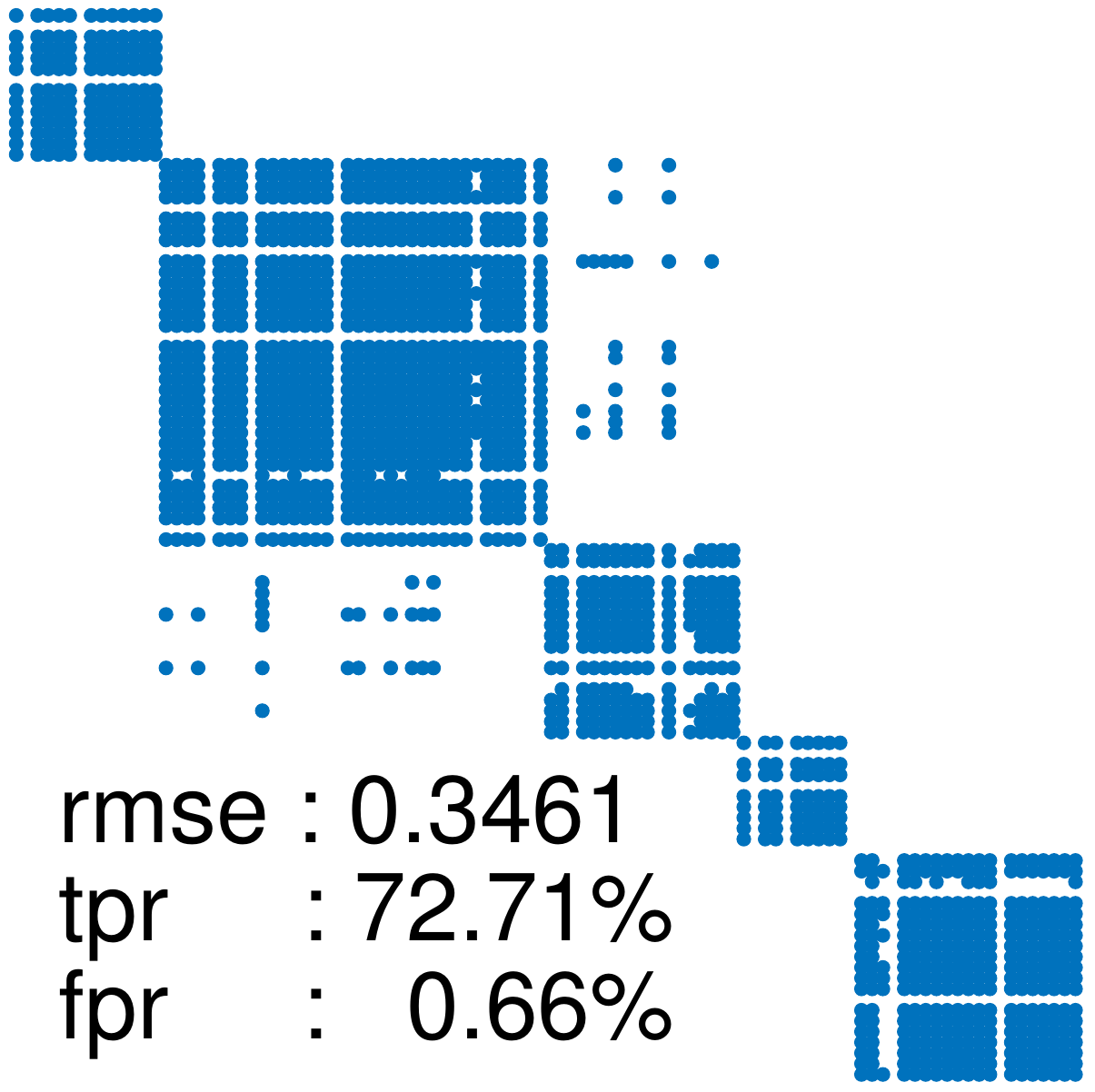}}\\
\subfloat[$\vY^*$]{\label{fig:true300}\includegraphics[scale=0.2]{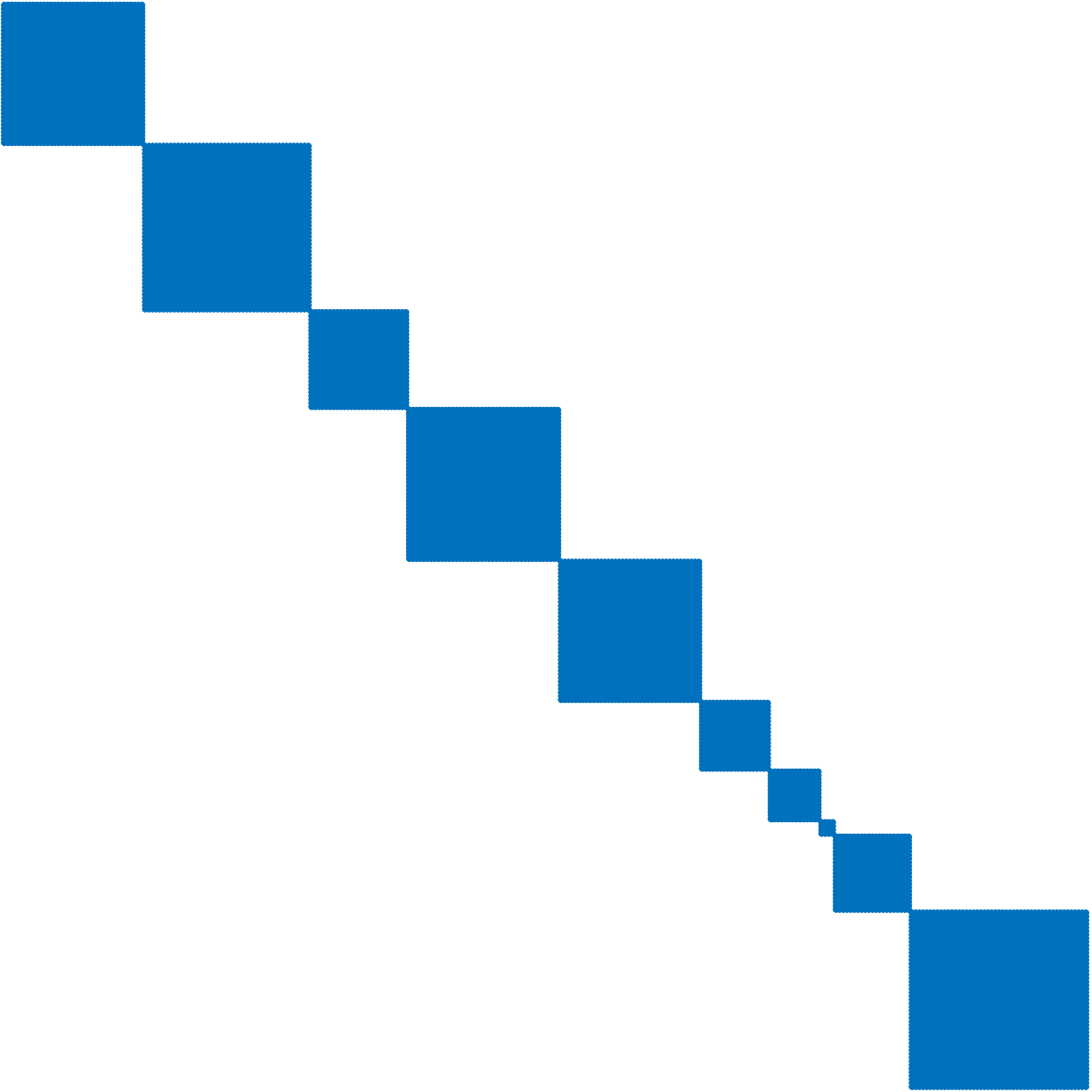}}\qquad \subfloat[DR]{\label{fig:DR300}\includegraphics[scale=0.2]{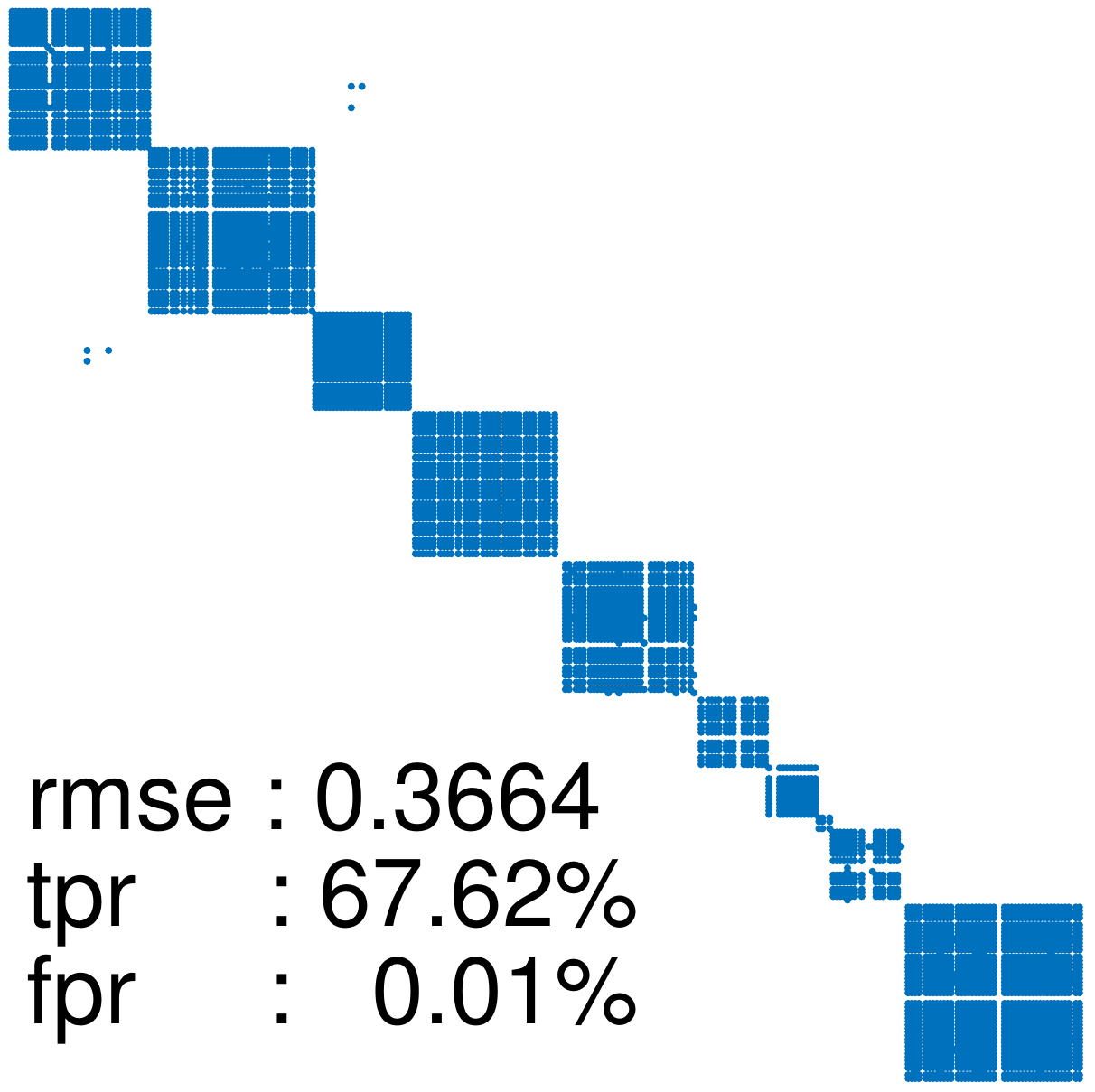}}\qquad \subfloat[ADMM]{\label{fig:ADMM300}\includegraphics[scale=0.2]{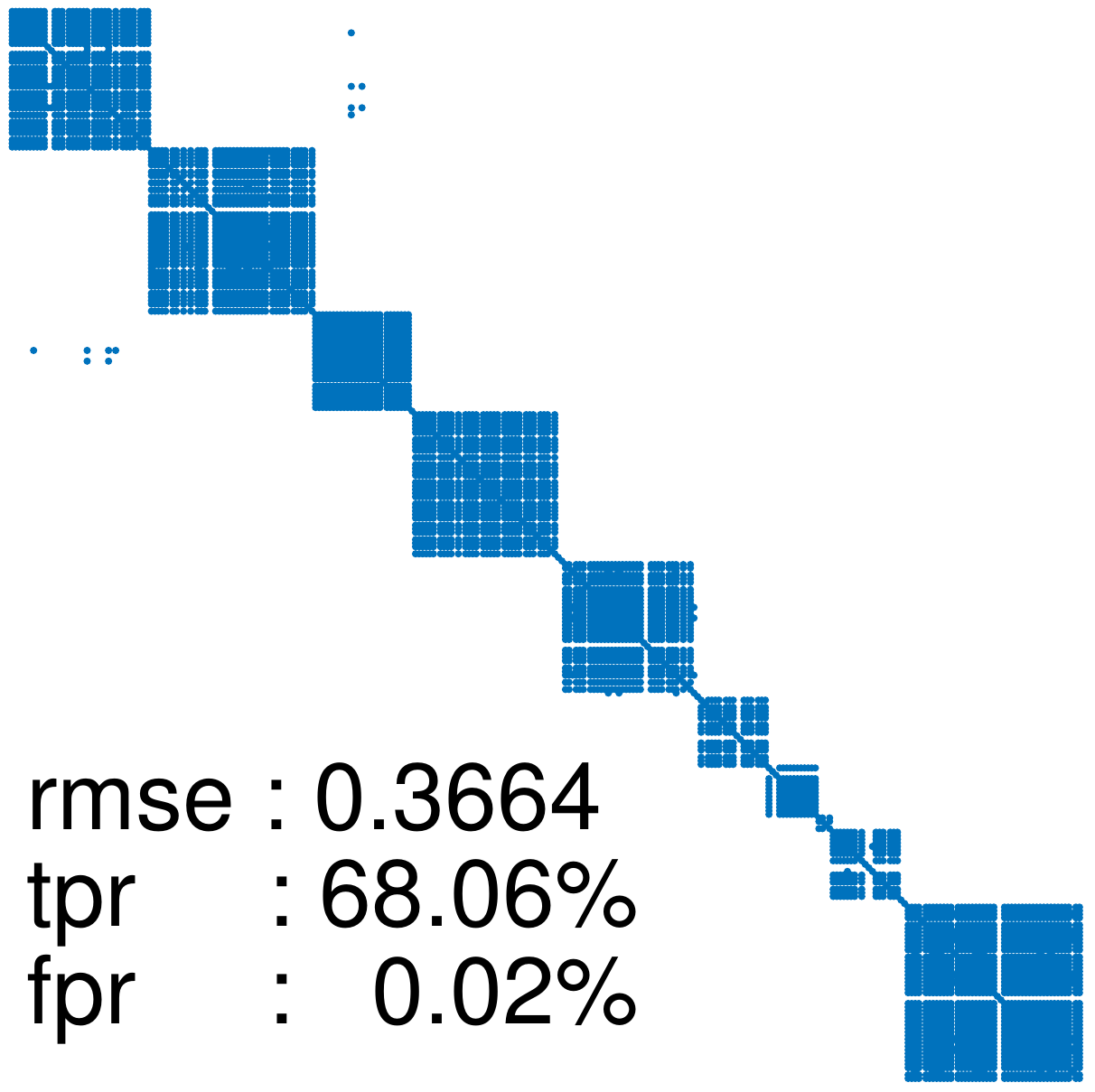}}\qquad \subfloat[IPD]{\label{fig:IPD300}\includegraphics[scale=0.2]{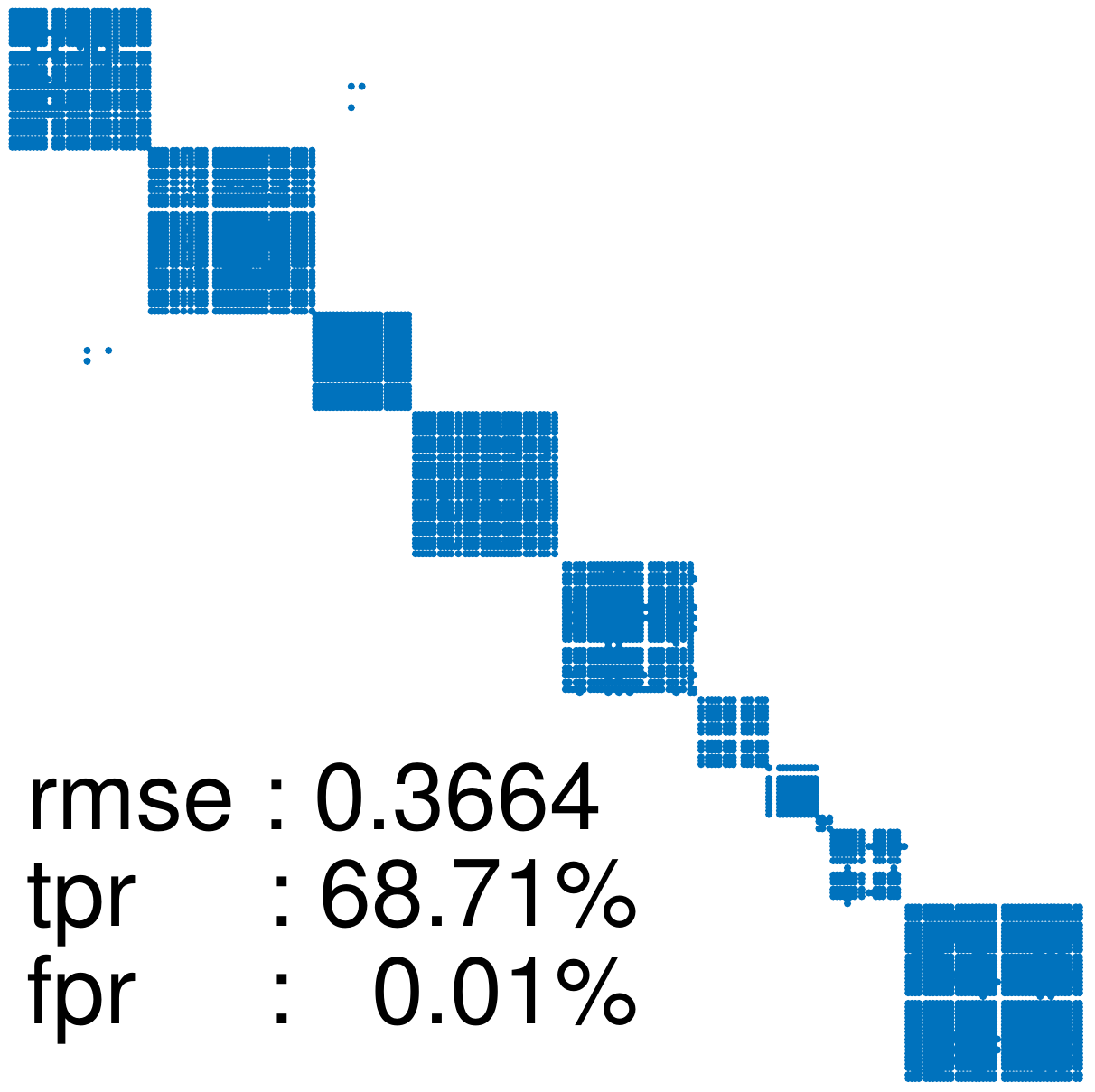}}\\
\end{center}
\caption{Original matrix and reconstruction results for DR, ADMM and IPD algorithms, for $n=100$ (top) and $n=300$ (bottom).}
\label{fig:qualcomp}
\end{figure}

\autoref{fig:qualcomp} illustrates the quality of the recovered covariance matrices when setting $\varepsilon=10^{-10}$. Three different indicators for estimation quality are provided, namely the \emph{true positive rate }(\verb"tpr"), i.e. the correctly recognized non--zero entries,  the \emph{false positive rate} (\verb"fpr"), i.e. the entries erroneously added to the support of the matrix, and the \emph{relative mean square error} (\verb"rmse"), computed as $\|\vY_{\rm rec}-\vY^*\|_{\rm F}^2/\|\vY^*\|_{\rm F}^2$, with $\vY_{\rm rec}$ the recovered matrix. Note that the two first measurements are employed when the main interest lies in the recovery of the matrix support. A visual inspection shows that the three methods provide similar results in terms of matrix support estimation. Moreover, the reconstruction error as well as the values of \verb"fpr" and \verb"tpr" slightly differ. 

\begin{table}[htbp]
\caption{Comparison in terms of convergence speed between DR, ADMM and IPD procedures. The enlighten times refer to the shortest ones.}
\label{tab:perfcomp}
\begin{center}
\small

\begin{tabular}{l | r@{\,}r  | r@{\,}r  | r@{\,}r  || r@{\,}r |  r@{\,}r | r@{\,}r}
& \multicolumn{6}{l||}{\tiny$n = 100$, $\mu_0=0.2,\mu_1=0.1$,$r =5$} & \multicolumn{5}{l}{\tiny$n = 300$, $\mu_0=0.01,\mu_1=0.12$}\\
& \multicolumn{6}{l||}{\tiny$\{r_j\}=\{14, 36,18,10, 22\}$} & \multicolumn{6}{l}{\tiny$r =10$, $\{r_j\}=\{39,    46,    27,    42,    39,    19,    14 ,    4    ,21 ,   49\}$}\\
\midrule
& \multicolumn{2}{c|}{DR} & \multicolumn{2}{c|}{ADMM}  & \multicolumn{2}{c||}{IPD} &\multicolumn{2}{c|}{DR} & \multicolumn{2}{c|}{ADMM}  & \multicolumn{2}{c}{IPD} \\
\toprule
\multicolumn{1}{c|}{$\varepsilon$} & \multicolumn{2}{c|}{Time(iter)} &  \multicolumn{2}{c|}{Time(iter)} & \multicolumn{2}{c||}{Time(iter)}  & \multicolumn{2}{c|}{Time(iter)} &  \multicolumn{2}{c|}{Time(iter)} & \multicolumn{2}{c}{Time(iter)}  \\
\midrule

$10^{-6}$ 	& 0.03&(23)  				& \textbf{0.02}&(17) 		& 0.18&(167)		& 0.14&(17)  				& \textbf{0.11}&(14) 					& 1.34		&(170)		\\
$10^{-7}$ 	& 0.03&(27)  					& \textbf{0.02}&(21) 		& 0.58&(533)		& \textbf{0.32}&(38)  	& 0.34&(42) 					& 4.35		&(548)		\\
$10^{-8}$ 	& \textbf{0.03}&(30)  	& 0.04&(34) 					& 1.83&(685)		& \textbf{0.81}&(95)  	& 0.91&(115) 				& 13.72	&(1748)	\\
$10^{-9}$ 	& \textbf{0.06}&(56)  	& \textbf{0.06}&(54) 	& 2.16&(2000)		& \textbf{1.79}&(211)  & 2.06&(258) 				& 15.70	&(2000)	\\
$10^{-10}$ & \textbf{0.07}&(59)  	& \textbf{0.07}&(58) 	& 2.16&(2000)		& \textbf{5.23}&(620)  & 5.45&(686) 				& 15.68	&(2000)	\\
\midrule
\end{tabular}
\end{center}
\end{table}

\autoref{tab:perfcomp} presents the comparative performance of the algorithms in terms of computation time (in second) and iteration number (averaged on 20 noise realizations), for two scenarios corresponding to distinct problem sizes and block distributions. It can be observed that the behaviors of ADMM and DR are similar, while IPD requires more iterations and time to reach the same precision. Furthermore, the latter fails to reach a high precision in the  allowed maximum number of iterations, for both examples.

\subsection{Application to Robust Graphical Lasso}
\label{sub:noisy}
Let us now illustrate the applicability of the MM approach presented in \autoref{sub:algo} to the problem of precision matrix estimation introduced in \eqref{eq:prec}. 
The test datasets have been generated by using the code available at \url{http://stanford.edu/ ̃boyd/papers/admm/covsel/ covsel_example.html}. A sparse precision matrix $\vC^*$ of dimension $n \times n$ is randomly created, where the number of non--zero entries is chosen as a proportion $p\in ]0,1[$ of the total number $n^2$. Then, $N$ realizations $(\vx^{(i)})_{1\le i\le N}$ of a Gaussian multivalued random variable with zero mean and covariance $\vY^* = (\vC^*)^{-1}$ are generated. Gaussian noise with zero mean and covariance $\sigma^2 \Id$, $\sigma>0$, is finally added to the $\vx^{(i)}$'s, so that the covariance matrix $\bm\Sigma$ associated with the input data reads as in \eqref{eq:xcor} with $\vA = \Id$. As explained in \autoref{sub:model}, the estimation of $ \vC^*$ can be performed by using the MM algorithm from \autoref{sub:algo} based on the minimization of the nonconvex cost \eqref{eq:prec} with regularization functions $g_1 = \mu_1\|\cdot\|_1$, $\mu_1>0$, and $(\forall \vC \in \smp)$ $g_0(\vC)=\mu_0\Sp{1}{\vC^{-1}}$, $\mu_0>0$. The computation of $\textnormal{prox}_{\gamma\lp \varphi + \psi\rp}$ with $\gamma \in ]0,+\infty[$ related to this particular choice for $g_0$ and 
function $\varphi$ given by \eqref{eq:spectralnewlog} and \eqref{eq:u} leads to the search of the only positive root of a polynomial of degree 4.

A synthetic dataset of size $n=100$ is created, where matrix $\vC^*$ has 20 off-diagonal non-zero entries (i.e., $p=10^{-3}$) and the corresponding covariance matrix has condition number 0.125. $N=1000$ realizations are used to compute the empirical covariance matrix $\vS$.  In our MM algorithm, the inner stopping criterion (line \ref{s:stopalg2} in Algorithm \ref{al:MMDR}) is based on the relative difference of majorant function values with a tolerance of $10^{-10}$, while the outer cycle is stopped when the relative difference of the objective function values falls below $10^{-8}$. The DR algorithm is used to solve the inner subproblems, by using parameters $(\forall \ell)$ $\gamma_\ell=1$, $(\forall k)$ $\alpha_{\ell,k}=1$ (see Algorithm \ref{al:MMDR}, lines~\ref{e:startDRalg2}--\ref{e:stopDRalg2}). The allowed  maximum inner (resp. outer) iteration number is 2000 (resp. 20). The quality of the results is quantified in terms of \verb"fpr" on the precision matrix and \verb"rmse" with respect to the true covariance matrix. The parameters $\mu_1$ and $\mu_0$ are set in order to obtain the best reconstruction in terms of \verb"rmse". For eight values of the noise standard deviation $\sigma$, \autoref{fig:sigma8} illustrates the reconstruction quality (averaged on $20$ noise realizations)  obtained with our method, as well as two other approaches that do not take into account the noise in their formulation, namely the classical GLASSO approach from \cite{Boyd:2011:DOS:2185815.2185816}, which amounts to solve \eqref{eq:prob} with $f=-\log\det, \,g=\mu_1\|\cdot\|_1$, and the DR approach described in \autoref{sec:DR}, in the formulation given by \eqref{eq:prob} with  $f=-\log\det$, $(\forall \vC \in \smp)$ $g(\vC)=\mu_0\Sp{1}{\vC^{-1}}+\mu_1\|\vC\|_1$. For the DR approach, 
$\textnormal{prox}_{\gamma\lp \varphi + \psi\rp}$ with $\gamma \in ]0,+\infty[$ is given by the fourth line of Table \ref{tab:logdet} (when $p=1$).
\begin{figure}[htbp]
\begin{center}
\subfloat[Behaviour of rmse wrt $\sigma$.]{\label{fig:sigma8a}\includegraphics[scale=0.3]{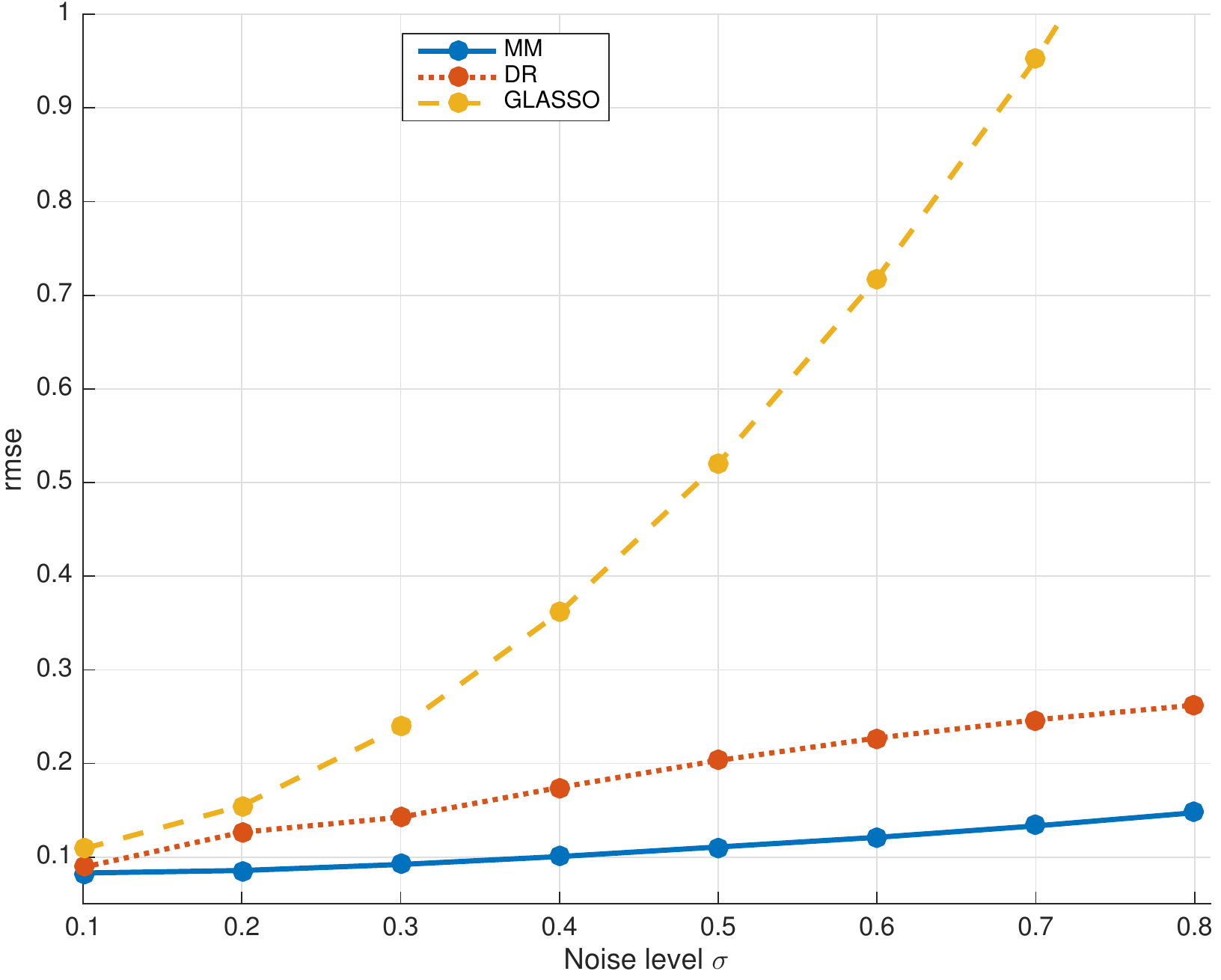}}\qquad \subfloat[Behaviour of fpr wrt $\sigma$.]{\label{fig:sigma8b}\includegraphics[scale=0.3]{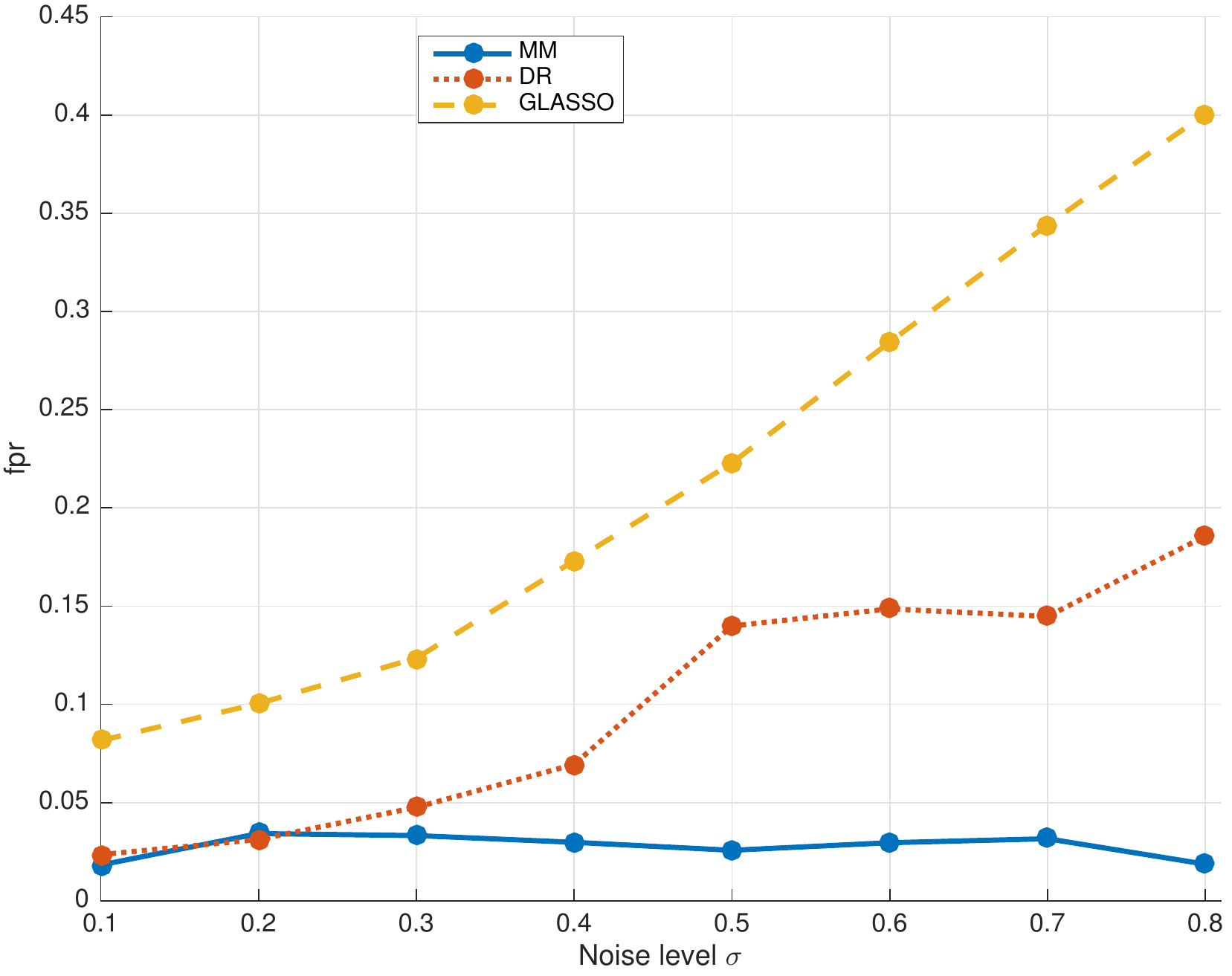}}\qquad 
\end{center}
\caption{Estimation results for different noise levels in terms of \texttt{rmse} (left) and \texttt{fpr} (right) for MM, GLASSO and DR approaches.}
\label{fig:sigma8}
\end{figure}
As expected, as the noise variance increases the reconstruction quality deteriorates. The GLASSO procedure is strongly impacted by the presence of noise, whereas the MM approach achieves better results, also when compared with DR algorithm. Moreover, the MM algorithm significantly outperforms both other methods in terms of support reconstruction, revealing itself very robust with respect to an increasing level of noise.

\section{Conclusions}
\label{sec:conclusions}
In this work, various proximal tools have been introduced to deal with optimization problems involving real symmetric matrices. 
We have focused on the variational framework \eqref{eq:prob} which is closely 
related to the computation of a proximity operator with respect to a Bregman divergence.
It has been assumed that $f$ in \eqref{eq:breg} is a convex spectral function, and $g$ reads as $g_0+g_1$, where $g_0$ is a spectral function. We have given a fully spectral solution in \autoref{sec:Dec} when $g_1\equiv 0$, and, in particular, \autoref{co:proxdiv} could be useful for developing algorithms involving proximity operators in other metrics than the Frobenius one. When $g_1\not \equiv 0$, a proximal iterative approach has been presented, which is  grounded on the use of the Douglas--Rachford procedure.
As illustrated by the tables of proximity operators provided for a wide range of choices for $f$ and $g_0$, the main advantage of the proposed algorithm is its great flexibility.
The proposed framework also has allowed us to propose a nonconvex formulation of the 
precision matrix estimation problem arising in the context of noisy graphical lasso.
The nonconvexity of the obtained objective function has been cirmcumvented through 
a Majorization--Minimization approach, each step of which consists of solving a convex problem by a Douglas-Rachford sub-iteration.

Comparisons with state--of--the--art solutions have demonstrated the robustness of the proposed method.

It is worth mentioning that all the results presented in this paper can be easily extended to  complex Hermitian matrices.

\bibliographystyle{spmpsci}      
\bibliography{abbr,biblio}

\begin{thebibliography}{10}
\providecommand{\url}[1]{{#1}}
\providecommand{\urlprefix}{URL }
\expandafter\ifx\csname urlstyle\endcsname\relax
  \providecommand{\doi}[1]{DOI~\discretionary{}{}{}#1}\else
  \providecommand{\doi}{DOI~\discretionary{}{}{}\begingroup
  \urlstyle{rm}\Url}\fi

\bibitem{Arago2013}
Arag{\'o}n~Artacho, F.J., Borwein, J.M.: Global convergence of a non-convex
  {D}ouglas--{R}achford iteration.
\newblock J. Global Optim. \textbf{57}(3), 753--769 (2013).
\newblock \doi{10.1007/s10898-012-9958-4}

\bibitem{Aslan2016}
Aslan, M.S., Chen, X.W., Cheng, H.: Analyzing and learning sparse and
  scale-free networks using {G}aussian graphical models.
\newblock J. Mach. Learn. Res. \textbf{1}(2), 99--109 (2016).
\newblock \doi{10.1007/s41060-016-0009-y}

\bibitem{Banerjee:2008:MST:1390681.1390696}
Banerjee, O., El~Ghaoui, L., d'Aspremont, A.: Model selection through sparse
  maximum likelihood estimation for multivariate {G}aussian or binary data.
\newblock J. Mach. Learn. Res. \textbf{9}, 485--516 (2008)

\bibitem{Bauscke-Borwein01}
Bauschke, H.H., Borwein, J.M., Combettes, P.L.: Essential smoothness, essential
  strict convexity, and {L}egendre functions in {B}anach spaces.
\newblock Comm. Contemp. Math \textbf{3}, 615--647 (2001)

\bibitem{Bauschke03}
Bauschke, H.H., Borwein, J.M., Combettes, P.L.: Bregman monotone optimization
  algorithms.
\newblock SIAM J. Control Optim. \textbf{42}(2), 596--636 (2003).
\newblock \doi{10.1137/S0363012902407120}

\bibitem{Bauschke:2017}
Bauschke, H.H., Combettes, P.L.: Convex Analysis and Monotone Operator Theory
  in {H}ilbert Spaces, 2nd edn.
\newblock Springer International Publishing (2017).
\newblock \doi{10.1007/978-3-319-48311-5}

\bibitem{Bauschke06}
Bauschke, H.H., Combettes, P.L., Noll, D.: Joint minimization with alternating
  {B}regman proximity operators.
\newblock Pac. J. Optim. \textbf{2}(3), 401--424 (2006)

\bibitem{Benfenati13}
Benfenati, A., Ruggiero, V.: Inexact {B}regman iteration with an application to
  {P}oisson data reconstruction.
\newblock Inverse Problems \textbf{29}(6), 1--32 (2013)

\bibitem{Benfenati2015}
Benfenati, A., Ruggiero, V.: Inexact {B}regman iteration for deconvolution of
  superimposed extended and point sources.
\newblock Commun. Nonlinear Sci. Numer. Simul. \textbf{20}(3), 882 -- 896
  (2015).
\newblock \doi{http://dx.doi.org/10.1016/j.cnsns.2014.06.045}

\bibitem{bengtsson2006}
Bengtsson, I., Zyczkowski, K.: Geometry of Quantum States: An Introduction to
  Quantum Entanglement.
\newblock Cambridge University Press, Cambridge (2006).
\newblock \doi{10.1017/CBO9780511535048}

\bibitem{doi:10.1137/080714488}
van~den Berg, E., Friedlander, M.P.: Probing the {P}areto frontier for basis
  pursuit solutions.
\newblock SIAM J. Sci. Comput. \textbf{31}(2), 890--912 (2009).
\newblock \doi{10.1137/080714488}

\bibitem{borwein2014}
Borwein, J., Lewis, A.: Convex Analysis and Nonlinear Optimization.
\newblock Springer (2014)

\bibitem{Boyd:2011:DOS:2185815.2185816}
Boyd, S., Parikh, N., Chu, E., Peleato, B., Eckstein, J.: Distributed
  optimization and statistical learning via the alternating direction method of
  multipliers.
\newblock Found. Trends Mach. Learn. \textbf{3}(1), 1--122 (2011).
\newblock \doi{10.1561/2200000016}

\bibitem{bregman1967}
Bregman, L.M.: {The Relaxation Method of Finding the Common Point of Convex
  Sets and Its Application to the Solution of Problems in Convex Programming}.
\newblock USSR Computational Mathematics and Mathematical Physics \textbf{7},
  200--217 (1967)

\bibitem{Brune2011}
Brune, C., Sawatzky, A., Burger, M.: Primal and dual {B}regman methods with
  application to optical nanoscopy.
\newblock Int. J. Comput. Vis. \textbf{92}(2), 211--229 (2011).
\newblock \doi{10.1007/s11263-010-0339-5}

\bibitem{Burger2016}
Burger, M., Sawatzky, A., Steidl, G.: First Order Algorithms in Variational
  Image Processing, pp. 345--407.
\newblock Springer International Publishing, Cham (2016).
\newblock \doi{10.1007/978-3-319-41589-5_10}

\bibitem{Cai10}
Cai, J.F., CandÃšs, E.J., Shen, Z.: A singular value thresholding algorithm
  for matrix completion.
\newblock SIAM J. Optim. \textbf{20}(4), 1956--1982 (2010).
\newblock \doi{10.1137/080738970}

\bibitem{doi:10.1198/jasa.2011.tm10155}
Cai, T., Liu, W., Luo, X.: A constrained $\ell_1$ minimization approach to
  sparse precision matrix estimation.
\newblock J. Am. Stat. Assoc. \textbf{106}(494), 594--607 (2011).
\newblock \doi{10.1198/jasa.2011.tm10155}

\bibitem{chandrasekaran2012}
Chandrasekaran, V., Parrilo, P.A., Willsky, A.S.: Latent variable graphical
  model selection via convex optimization.
\newblock Ann. Statist. \textbf{40}(4), 1935--1967 (2012).
\newblock \doi{10.1214/11-AOS949}

\bibitem{chartrand12}
Chartrand, R.: Nonconvex splitting for regularized low-rank + sparse
  decomposition.
\newblock IEEE Trans. Signal Process. \textbf{60}, 5810--5819 (2012)

\bibitem{0266-5611-23-4-008}
Chaux, C., Combettes, P.L., Pesquet, J.C., Wajs, V.R.: A variational
  formulation for frame-based inverse problems.
\newblock Inverse Problems \textbf{23}(4), 1495 (2007)

\bibitem{NellyPois}
Chaux, C., Pesquet, J.C., Pustelnik, N.: Nested iterative algorithms for convex
  constrained image recovery problem.
\newblock SIAM J. Imaging Sci. \textbf{2}(2), 730--762 (2009)

\bibitem{chouzenoux16}
Chouzenoux, E., Pesquet, J.C.: {Convergence Rate Analysis of the
  Majorize-Minimize Subspace Algorithm}.
\newblock IEEE Signal Process. Lett. \textbf{23}(9), 1284 -- 1288 (2016).
\newblock \doi{10.1109/LSP.2016.2593589}

\bibitem{combettes:hal-00621820}
Combettes, P.L., Pesquet, J.C.: {A {D}ouglas-{R}achford splitting approach to
  nonsmooth convex variational signal recovery}.
\newblock IEEE J. Sel. Topics Signal Process. \textbf{1}(4), 564--574 (2007)

\bibitem{combettes:hal-00643807}
Combettes, P.L., Pesquet, J.C.: {Proximal Splitting Methods in Signal
  Processing}.
\newblock In: {Fixed-Point Algorithms for Inverse Problems in Science and
  Engineering}, pp. 185--212. {Springer} (2011).
\newblock \doi{10.1007/978-1-4419-9569-8}

\bibitem{Condat2016}
Condat, L.: Fast projection onto the simplex and the $\ell_1$ ball.
\newblock Math. Programm. \textbf{158}(1), 575--585 (2016).
\newblock \doi{10.1007/s10107-015-0946-6}

\bibitem{Corless1996}
Corless, R.M., Gonnet, G.H., Hare, D.E.G., Jeffrey, D.J., Knuth, D.E.: {On the
  Lambert W function}.
\newblock Adv. Comput. Math. \textbf{5}(1), 329--359 (1996).
\newblock \doi{10.1007/BF02124750}

\bibitem{cover2006elements}
Cover, T., Thomas, J.: Elements of Information Theory.
\newblock A Wiley-Interscience publication. Wiley (2006)

\bibitem{Aspremont08}
d'Aspremont, A., Banerjee, O., Ghaoui, L.E.: First-order methods for sparse
  covariance selection.
\newblock SIAM J. Matrix Anal. Appl. \textbf{30}(1), 56--66 (2008).
\newblock \doi{10.1137/060670985}

\bibitem{Dempster72}
Dempster, A.: Covariance selection.
\newblock Biometrics \textbf{28}, 157--175 (1972)

\bibitem{DBLP:conf/uai/2008}
Duchi, J.C., Gould, S., Koller, D.: {Projected Subgradient Methods for Learning
  Sparse {G}aussians}.
\newblock In: {UAI} 2008, Proceedings of the 24th Conference in Uncertainty in
  Artificial Intelligence, Helsinki, Finland, July 9-12, 2008, pp. 145--152
  (2008)

\bibitem{Friedman07}
Friedman, J., Hastie, T., Tibshirani, R.: {Sparse inverse covariance estimation
  with the graphical lasso}.
\newblock Biostatistics \textbf{9}(3), 432--441 (2008).
\newblock \doi{10.1093/biostatistics/kxm045}

\bibitem{doi:10.1137/080725891}
Goldstein, T., Osher, S.: The split {B}regman method for l1-regularized
  problems.
\newblock SIAM J. Imaging Sci. \textbf{2}(2), 323--343 (2009).
\newblock \doi{10.1137/080725891}

\bibitem{doi:10.1093/biomet/asq060}
Guo, J., Levina, E., Michailidis, G., Zhu, J.: Joint estimation of multiple
  graphical models.
\newblock Biometrika \textbf{98}(1), 1 (2011).
\newblock \doi{10.1093/biomet/asq060}

\bibitem{hardy1952inequalities}
Hardy, G., Littlewood, J., P{\'o}lya, G.: Inequalities.
\newblock Cambridge Mathematical Library. Cambridge University Press (1952)

\bibitem{doi:10.1198/0003130042836}
Hunter, D.R., Lange, K.: A tutorial on {MM} algorithms.
\newblock Amer. Statist. \textbf{58}(1), 30--37 (2004).
\newblock \doi{10.1198/0003130042836}

\bibitem{Jacobson07}
Jacobson, M.W., Fessler, J.A.: An expanded theoretical treatment of
  iteration-dependent majorize-minimize algorithms.
\newblock IEEE Trans. Image Process. \textbf{16}(10), 2411--2422 (2007).
\newblock \doi{10.1109/TIP.2007.904387}

\bibitem{Komodakis15}
Komodakis, N., Pesquet, J.C.: Playing with duality: An overview of recent
  primal--dual approaches for solving large-scale optimization problems.
\newblock IEEE Signal Process. Mag. \textbf{32}(6), 31--54 (2015).
\newblock \doi{10.1109/MSP.2014.2377273}

\bibitem{Lewis96}
Lewis, A.S.: Convex analysis on the {H}ermitian matrices.
\newblock SIAM J. Optim. \textbf{6}(1), 164--177 (1996).
\newblock \doi{10.1137/0806009}

\bibitem{Li2016}
Li, G., Pong, T.K.: Douglas--{R}achford splitting for nonconvex optimization
  with application to nonconvex feasibility problems.
\newblock Math. Programm. \textbf{159}(1), 371--401 (2016).
\newblock \doi{10.1007/s10107-015-0963-5}

\bibitem{Li2010}
Li, L., Toh, K.C.: An inexact interior point method for $\ell_1$--regularized
  sparse covariance selection.
\newblock Math. Program. Comput. \textbf{2}(3), 291--315 (2010).
\newblock \doi{10.1007/s12532-010-0020-6}

\bibitem{doi:10.1137/0716071}
Lions, P.L., Mercier, B.: Splitting algorithms for the sum of two nonlinear
  operators.
\newblock SIAM J. Numer. Anal. \textbf{16}(6), 964--979 (1979).
\newblock \doi{10.1137/0716071}

\bibitem{doi:10.1137/070695915}
Lu, Z.: Smooth optimization approach for sparse covariance selection.
\newblock SIAM J. Optim. \textbf{19}(4), 1807--1827 (2009).
\newblock \doi{10.1137/070695915}

\bibitem{doi:10.1137/080742531}
Lu, Z.: Adaptive first-order methods for general sparse inverse covariance
  selection.
\newblock SIAM J. Matrix Anal. Appl. \textbf{31}(4), 2000--2016 (2010).
\newblock \doi{10.1137/080742531}

\bibitem{Ma:2013:ADM:2494250.2494257}
Ma, S., Xue, L., Zou, H.: Alternating direction methods for latent variable
  {G}aussian graphical model selection.
\newblock Neural Comput. \textbf{25}(8), 2172--2198 (2013).
\newblock \doi{10.1162/NECO_a_00379}

\bibitem{magnus99}
Magnus, J.R., Neudecker, H.: Matrix Differential Calculus with Applications in
  Statistics and Econometrics, second edn.
\newblock John Wiley (1999)

\bibitem{marshall11}
Marshall, A.W., Olkin, I., Arnold, B.C.: Inequalities: Theory of Majorization
  and its Applications, vol. 143, second edn.
\newblock Springer (2011).
\newblock \doi{10.1007/978-0-387-68276-1}

\bibitem{mazumder2012}
Mazumder, R., Hastie, T.: The graphical lasso: New insights and alternatives.
\newblock Electron. J. Stat. \textbf{6}, 2125--2149 (2012).
\newblock \doi{10.1214/12-EJS740}

\bibitem{meinshausen2006}
Meinshausen, N., Bühlmann, P.: High-dimensional graphs and variable selection
  with the lasso.
\newblock Ann. Statist. \textbf{34}(3), 1436--1462 (2006).
\newblock \doi{10.1214/009053606000000281}

\bibitem{Moreau65}
Moreau, J.: Proximité et dualité dans un espace hilbertien.
\newblock Bull. Soc. Math. France \textbf{93}, 273--299 (1965)

\bibitem{Nesterov2005}
Nesterov, Y.: Smooth minimization of non-smooth functions.
\newblock Math. Programm. \textbf{103}(1), 127--152 (2005).
\newblock \doi{10.1007/s10107-004-0552-5}

\bibitem{Parikh14}
Parikh, N., Boyd, S.: Proximal algorithms.
\newblock Found. Trends Optim. \textbf{1}(3), 127--239 (2014).
\newblock \doi{10.1561/2400000003}

\bibitem{pesquet:hal-00790702}
Pesquet, J.C., Pustelnik, N.: A parallel inertial proximal optimization method.
\newblock Pac. J. Optim. \textbf{8}(2), 273--305 (2012)

\bibitem{ravikumar2011}
Ravikumar, P., Wainwright, M.J., Raskutti, G., Yu, B.: High-dimensional
  covariance estimation by minimizing $\ell_1$-penalized log-determinant
  divergence.
\newblock Electron. J. Statist. \textbf{5}, 935--980 (2011).
\newblock \doi{10.1214/11-EJS631}

\bibitem{Richard:2012:ESS:3042573.3042584}
Richard, E., andre Savalle, P., Vayatis, N.: Estimation of simultaneously
  sparse and low rank matrices.
\newblock In: Proceedings of the 29th International Conference on Machine
  Learning (ICML-12), pp. 1351--1358. ACM (2012)

\bibitem{Rockfellar70}
Rockafellar, R.: Convex Analysis.
\newblock Princeton landmarks in mathematics and physics. Princeton University
  Press (1970)

\bibitem{Rockfellar97}
Rockafellar, R.T., Wets, R.J.B.: Variational Analysis, 1st edn.
\newblock Springer-Verlag (1997)

\bibitem{rothman2008}
Rothman, A.J., Bickel, P.J., Levina, E., Zhu, J.: Sparse permutation invariant
  covariance estimation.
\newblock Electron. J. Statist. \textbf{2}, 494--515 (2008).
\newblock \doi{10.1214/08-EJS176}

\bibitem{NIPS2010_0109}
Scheinberg, K., Ma, S., Goldfarb, D.: Sparse inverse covariance selection via
  alternating linearization methods.
\newblock In: Advances in Neural Information Processing Systems 23, pp.
  2101--2109 (2010)

\bibitem{sun17}
Sun, Y., Babu, P., Palomar, D.P.: Majorization-{M}inimization algorithms in
  signal processing, communications, and machine learning.
\newblock IEEE Trans. Signal Process. \textbf{65}(3), 794--816 (2017).
\newblock \doi{10.1109/TSP.2016.2601299}

\bibitem{Tipping01}
Tipping, M.E.: Sparse {B}ayesian learning and the relevance vector machine.
\newblock J. Mach. Learn. Res. \textbf{1}, 211--244 (2001).
\newblock \doi{10.1162/15324430152748236}

\bibitem{doi:10.1137/090772514}
Wang, C., Sun, D., Toh, K.C.: Solving log-determinant optimization problems by
  a {N}ewton-{CG} primal proximal point algorithm.
\newblock SIAM J. Optim. \textbf{20}(6), 2994--3013 (2010).
\newblock \doi{10.1137/090772514}

\bibitem{wipf04}
Wipf, D.P., Rao, B.D.: Sparse {B}ayesian learning for basis selection.
\newblock IEEE Trans. Signal Process. \textbf{52}(8), 2153--2164 (2004).
\newblock \doi{10.1109/TSP.2004.831016}

\bibitem{Wu_83}
Wu, C.F.J.: On the convergence properties of the {EM} algorithm.
\newblock Ann. Statist. \textbf{11}(1), 95--103 (1983).
\newblock \doi{10.1214/aos/1176346060}

\bibitem{doi:10.1137/130936397}
Yang, S., Lu, Z., Shen, X., Wonka, P., Ye, J.: Fused multiple graphical lasso.
\newblock SIAM J. Optim. \textbf{25}(2), 916--943 (2015).
\newblock \doi{10.1137/130936397}

\bibitem{Yin08}
Yin, W., Osher, S., Goldfarb, D., Darbon, J.: Bregman iterative algorithms for
  {$\ell_1$}-minimization with applications to compressed sensing.
\newblock SIAM J. Imaging Sci. \textbf{1}(1), 143--168 (2008).
\newblock \doi{10.1137/070703983}

\bibitem{doi:10.1093/biomet/asm018}
Yuan, M., Lin, Y.: Model selection and estimation in the {G}aussian graphical
  model.
\newblock Biometrika \textbf{94}(1), 19 (2007).
\newblock \doi{10.1093/biomet/asm018}

\bibitem{Yuan09}
Yuan, X.: Alternating direction methods for sparse covariance selection
  (2009).
\newblock
  \urlprefix\url{http://www.optimization-online.org/DBFILE/2009/09/2390.pdf}

\bibitem{Zangwill_69}
Zangwill, W.I.: Nonlinear programming : a unified approach.
\newblock Englewood Cliffs, N.J. : Prentice-Hall (1969)

\bibitem{doi:10.1137/090746379}
Zhang, X., Burger, M., Bresson, X., Osher, S.: Bregmanized nonlocal
  regularization for deconvolution and sparse reconstruction.
\newblock SIAM J. Imaging Sci. \textbf{3}(3), 253--276 (2010).
\newblock \doi{10.1137/090746379}

\bibitem{Zhang2011}
Zhang, X., Burger, M., Osher, S.: A unified primal-dual algorithm framework
  based on {B}regman iteration.
\newblock J. Sci. Comput. \textbf{46}(1), 20--46 (2011).
\newblock \doi{10.1007/s10915-010-9408-8}

\bibitem{Zhou14}
{Zhou}, S., {Xiu}, N., {Luo}, Z., {Kong}, L.: Sparse and low-rank covariance
  matrices estimation  (2014)

\end{thebibliography}

\end{document}